\documentclass[11pt]{amsart}
\usepackage[margin=1.35in]{geometry}
\usepackage{amscd,amsmath,amsxtra,amsthm,amssymb,stmaryrd,xr,mathrsfs,mathtools,enumerate,commath, comment}
\usepackage{stmaryrd}
\usepackage{xcolor}
\usepackage{commath}
\usepackage{comment}
\usepackage{tikz-cd}
\usepackage{float}
\usepackage{longtable} 
\usepackage{pdflscape} 
\usepackage{booktabs}
\usepackage{hyperref}
\definecolor{vegasgold}{rgb}{0.77, 0.7, 0.35}
\definecolor{darkgoldenrod}{rgb}{0.72, 0.53, 0.04}
\definecolor{gold(metallic)}{rgb}{0.83, 0.69, 0.22}
\hypersetup{
 colorlinks=true,
 linkcolor=darkgoldenrod,
 filecolor=brown,      
 urlcolor=gold(metallic),
 citecolor=darkgoldenrod,
 pdftitle={An analogue of Kida's formula in graph theory},
 }

\usepackage[all,cmtip]{xy}

\usetikzlibrary{shapes.geometric}
\usetikzlibrary{decorations.markings}
\tikzset{every loop/.style={min distance=10mm,looseness=10}}

\DeclareFontFamily{U}{wncy}{}
\DeclareFontShape{U}{wncy}{m}{n}{<->wncyr10}{}
\DeclareSymbolFont{mcy}{U}{wncy}{m}{n}
\DeclareMathSymbol{\Sh}{\mathord}{mcy}{"58}
\usepackage[T2A,T1]{fontenc}
\usepackage[OT2,T1]{fontenc}

\newtheorem{theorem}{Theorem}[section]

\newtheorem{proposition}[theorem]{Proposition}

\newtheorem{definition}[theorem]{Definition}
\newtheorem{assumption}[theorem]{Assumption}

\numberwithin{equation}{section}

\newtheorem{lthm}{Theorem} 

\theoremstyle{remark}
\newtheorem{remark}[theorem]{Remark}
\newtheorem{example}[theorem]{Example}

\newcommand{\Y}{\mathcal{Y}}
\newcommand{\G}{\mathfrak{G}}

\newcommand{\Z}{\mathbb{Z}}

\newcommand{\Q}{\mathbb{Q}}

\newcommand{\cY}{\mathcal{Y}}

\newcommand{\op}[1]{\operatorname{#1}}

\newcommand\mtx[4] { \left( {\begin{array}{cc}
 #1 & #2 \\
 #3 & #4 \\
 \end{array} } \right)}

\begin{document}
\title[An analogue of Kida's formula in graph theory]{An analogue of Kida's formula in graph theory}

\author[A.~Ray]{Anwesh Ray}
\address[Ray]{Chennai Mathematical Institute, H1, SIPCOT IT Park, Kelambakkam, Siruseri, Tamil Nadu 603103, India}
\email{ar2222@cornell.edu}

\author[D.~Valli\`{e}res]{Daniel Valli\`{e}res}
\address{Daniel Valli\`{e}res\newline Mathematics and Statistics Department, California State University, Chico, CA 95929, USA}
\email{dvallieres@csuchico.edu}

\begin{abstract}
Let $\ell$ be a rational prime and let $p:Y\rightarrow X$ be a Galois cover of finite graphs whose Galois group is a finite $\ell$-group. Consider a $\mathbb{Z}_{\ell}$-tower above $X$ and its pullback along $p$. Assuming that all the graphs in the pullback are connected, one obtains a $\mathbb{Z}_{\ell}$-tower above $Y$. 
We prove a formula relating the Iwasawa $\lambda$-invariant of the $\mathbb{Z}_{\ell}$-tower above $X$ to the Iwasawa $\lambda$-invariant of the pullback.  This formula is analogous to Kida's formula in classical Iwasawa theory. We also study structural properties of certain noncommutative pro-$\ell$ towers of graphs, based on an analogy with classical results of Cuoco on the growth of Iwasawa invariants in $\Z_\ell^2$-extensions of number fields. Our investigations are illustrated by explicit examples. 
\end{abstract}

\date{\today} 

\maketitle
\tableofcontents 
 
\section{Introduction}
Let $f:\mathcal{S}_{1} \rightarrow \mathcal{S}_{2}$ be a branched cover of compact Riemann surfaces of degree $d$.  The Riemann-Hurwitz formula stipulates that
\begin{equation} \label{RH}
\chi(\mathcal{S}_{1}) = d\chi(\mathcal{S}_{2}) - \sum_{P \in \mathcal{S}_{1}}(e_{P} - 1), 
\end{equation}
where $\chi$ denotes the Euler characteristic of a compact Riemann surface, and $e_{P}$ is the ramification index at the point $P$. Note that the above sum is finite since $e_P=1$ for all but finitely many points $P$. If the cover is unramified, then \eqref{RH} simply becomes
\begin{equation} \label{RH_unramified}
\chi(\mathcal{S}_{1}) = d\chi(\mathcal{S}_{2}). 
\end{equation}
Using the relation $\chi(\mathcal{S}) = 2 - 2g(\mathcal{S})$, where $g(\mathcal{S})$ denotes the genus of $\mathcal{S}$, \eqref{RH} can also be written as
\begin{equation} \label{RH_genus}
2g(\mathcal{S}_{1}) - 2 = d(2g(\mathcal{S}_{2}) - 2) +  \sum_{P \in \mathcal{S}_{1}}(e_{P} - 1).
\end{equation}

\par Kida proved an analogous formula in the context of Iwasawa theory, which we now explain. Throughout, $\ell$ will be a rational prime,  and let $\mu_{\ell^\infty}$ be the $\ell$-power roots of unity in $\overline{\Q} \subseteq \mathbb{C}$. Let $K$ be a $CM$ number field. The cyclotomic $\Z_\ell$-extension of $K$ is the unique $\Z_\ell$-extension contained in $K(\mu_{\ell^\infty})$ and is denoted by $K_\infty$. Associated to $K_{\infty}$ is a tower of $CM$-fields
$$K = K_{0} \subseteq K_{1} \subseteq \ldots \subseteq K_{n} \subseteq \ldots $$
for which $K_{n}/K$ is Galois with Galois group isomorphic to $\mathbb{Z}/\ell^{n}\mathbb{Z}$. Let $K_n^+$ be the maximal totally real subfield of $K_n$ and $h_n^+$ be the class number of $K_n^+$. Denote by $h_n$, the class number of $K_n$ and let $h_n^-:=h_n/h_n^+$ denote the relative class number which is known to be an integer. Iwasawa showed (see \cite{Iwasawa:1959} and \cite{iwasawa1973zl}) that there exist nonnegative integers $\mu^{-} = \mu_{K}^{-},\lambda^{-} =  \lambda_{K}^{-}$ and an integer $\nu^{-} = \nu_{K}^{-}$ such that for $n$ large, one has
\begin{equation} \label{Iwasawa}
{\rm ord}_{\ell}(h_{n}^{-}) = \mu^{-}  \ell^{n} + \lambda^{-} n + \nu^{-}, 
\end{equation}
where ${\rm ord}_{\ell}$ is the usual $\ell$-adic valuation on $\mathbb{Q}$.  Let now $E/F$ be a Galois extension of $CM$-fields for which ${\rm Gal}(E/F)$ is an $\ell$-group, and assume for simplicity that $\ell > 2$ and that $F$ contains the $\ell$-th roots of unity.  Kida showed in \cite{Kida:1980} that if $\mu_{F}^{-} = 0$, then $\mu_{E}^{-} = 0$, and 
\begin{equation} \label{kida}
2\lambda_{E}^{-} - 2 = [E_{\infty}:F_{\infty}](2 \lambda_{F}^{-} - 2) + \sum_{w}(e_{w} - 1),
\end{equation}
where the sum is over all nonarchimedean places of $E_{\infty}$ that split completely in $E_{\infty}/E_{\infty}^{+}$ and that do not lie above $\ell$.  Several authors pointed out the striking similarity between Kida's formula (\ref{kida}) and the Riemann-Hurwitz formula (\ref{RH_genus}).

In this paper, we study analogies between Galois theoretic objects in arithmetic and their topological counterparts in graph theory. The properties of the Dedekind zeta function of a number field correlate to the arithmetic distribution properties of its primes. On the other hand, the Ihara zeta function in graph theory behaves much like the Dedekind zeta function, and has many interesting combinatorial applications, cf. \cite{Terras:2011}. There are also many analogies between graphs and algebraic curves. For instance, \cite{Baker/Norine:2009} establishes an analogue of the Abel-Jacobi map and the Riemann-Roch theorem in graph theory. There is a natural notion of the class group in graph theory (known in the literature under various names such as the critical group, the sandpile group, the Jacobian group, the Picard group, etc), which exhibits properties similar to its number theoretic and geometric counterparts. From an Iwasawa theoretic viewpoint, it is of natural interest to study the behavior of class numbers in various towers of graphs. Recently, a result similar to Iwasawa's theorem (\ref{Iwasawa}) above has been observed in graph theory. Let $X$ be a connected graph (with loops and parallel edges allowed), and consider a $\mathbb{Z}_{\ell}$-tower above $X$.  This means that we are given a sequence of covers of connected graphs
\begin{equation} \label{too}
X = X_{0} \leftarrow X_{1} \leftarrow \ldots \leftarrow X_{n} \leftarrow \ldots
\end{equation}
for which the cover $X_{n}/X$ obtained by composing the covers $X_{n} \rightarrow \ldots \rightarrow X_{1} \rightarrow X$ is Galois with group of covering transformations isomorphic to $\mathbb{Z}/\ell^{n}\mathbb{Z}$. We note that after the announcement of this article on the arXiv, Gambheera and the second author \cite{GambheeraVallieres} developed an Iwasawa theory for branched \(\mathbb{Z}_\ell\)-towers of graphs. In contrast, the present work focuses exclusively on unbranched Galois covers.  Then, \cite[Theorem 6.1]{mcgownvallieresIII} shows that there exist nonnegative integers $\mu,\lambda$ and an integer $\nu$ such that for $n$ large, one has
$${\rm ord}_{\ell}(\kappa_{n}) = \mu \ell^{n} + \lambda n + \nu, $$
where $\kappa_{n}$ is the class number (i.e., number of spanning trees) of the graph $X_{n}$. The approach is based on studying the properties of a suitable analogue of the $\ell$-adic $L$-function in graph theory, obtained from the $\ell$-adic interpolation of the special values of the Artin-Ihara $L$-functions. For a more algebraic approach to this result, one is referred to \cite{Gonet:2021a, Gonet:2022}.  We shall refer to $\mu$ and $\lambda$ as the Iwasawa invariants of the tower (\ref{too}).  Furthermore, it is known (see for instance \cite[page 55]{Sunada:2013}) that if $Y \rightarrow X$ is a cover of connected graphs of degree $d$, then one has
\begin{equation} \label{RH_graph}
\chi(Y) = d \chi(X),
\end{equation}
where $\chi$ here denotes the Euler characteristic of a graph.  It becomes then natural to study whether an analogue of Kida's formula holds true in graph theory as well.  Our main result is the following theorem.
\begin{lthm}[{Theorem \ref{main}}]\label{thmA}
Let $X$ be a connected graph for which $\chi(X) \neq 0$ and consider an unbranched $\mathbb{Z}_{\ell}$-tower
\begin{equation} \label{to}
X = X_{0} \leftarrow X_{1} \leftarrow \ldots \leftarrow X_{n} \leftarrow \ldots
\end{equation}
Let $p:Y \rightarrow X$ be a Galois cover of graphs whose Galois group is an $\ell$-group and pull back the tower (\ref{to}) along $p$ to obtain another $\mathbb{Z}_{\ell}$-tower 
\begin{equation} \label{to_pb}
Y \leftarrow Y \times_{X}X_{1} \leftarrow Y \times_{X}X_{2} \leftarrow \ldots \leftarrow Y \times_{X}X_{n} \leftarrow \ldots 
\end{equation}
Assume that all the graphs in both towers are connected.  Further, let us denote the Iwasawa invariants of the tower (\ref{to}) by $\mu_{X}$ and $\lambda_{X}$ and the Iwasawa invariants of the tower (\ref{to_pb}) by $\mu_{Y}$ and $\lambda_{Y}$.  Then $\mu_{X} = 0$ if and only if $\mu_{Y} = 0$ in which case
\begin{equation} \label{kida_gr}
\lambda_{Y}+1 = [Y:X] (\lambda_{X} +1), 
\end{equation}
where $[Y:X]$ is the degree of the cover $Y/X$.
\end{lthm}
Theorem \ref{thmA} can be viewed as an analogue of Kida's formula (\ref{kida}), but in graph theory.  We are aware of three different proofs of Kida's formula:  the original one by Kida in \cite{Kida:1980} using genus theory, the one by Iwasawa in \cite{Iwasawa:1981} where he proves a result analogous to the Chevalley-Weil formula for a branched Galois cover of compact Riemann surfaces, and the one by Sinnott in \cite{Sinnott:1984} using $\ell$-adic $L$-functions.  Our approach here to proving Theorem \ref{thmA} is inspired by \cite{Sinnott:1984}. The results in this article only apply to unbranched covers and have been recently generalized to branched covers in the setting of \cite{GambheeraVallieres}, by Kataoka \cite{Kataokaramifiedkida}. This proof however uses different methods from ours.

\par While the result certainly strengthens the analogy between classical Iwasawa theory and the Iwasawa theory of graphs, it is natural to ask whether this is the sole motivation or if there are other applications in mind. One possible motivation is that Kida’s formula in number theory plays a crucial role in understanding the behavior of Iwasawa invariants in larger Galois extensions. In \S \ref{section 5}, we present an application of Theorem \ref{main} to the study of certain growth questions in noncommutative pro-$\ell$ towers of graphs. Our investigations are inspired by a result of Cuoco \cite{cuoco1980growth}, who studied the growth of Iwasawa invariants in a $\Z_\ell^2$-tower over a number field $K$. This provides us insight into the structural properties of certain non-commutative pro-$\ell$ examples. These computations suggest that the Kida formula could serve as a stepping stone to further investigations in noncommutative Iwasawa theory. The authors believe that analogous investigations in the Iwasawa theory of graphs will lead to significant developments in combinatorics.
\begin{remark}
In \cite{Baker/Norine:2009}, Baker and Norine prove that if $\pi:Y\rightarrow X$ is a branched cover of graphs of degree $d$ in the sense of \cite[Notes 5.6]{Sunada:2013}, then
\begin{equation} \label{RH_graph_ram}
\chi(Y) = d\chi(X)  - \sum_{w \in V_{Y}}(e_{w} - 1),
\end{equation}
where $e_{w}$ is the ramification index of the vertex $w$.  In the present paper (as well as in \cite{Vallieres:2021, mcgownvallieresII, mcgownvallieresIII}, \cite{lei2022non}, \cite{Sage/Vallieres:2022}, \cite{DLRV}, and \cite{Gonet:2021a, Gonet:2022}), we restrict ourselves to unramified covers; therefore, the Euler characteristic satisfies the simpler formula (\ref{RH_graph}), and this might explain the lack of ramification indices in (\ref{kida_gr}).
\end{remark}
\emph{Organization:} In \S \ref{preliminaries}, we gather together a few previous results that will be used later on in the article.  We remind the reader about graphs and covers of graphs in \S \ref{gr}, about Artin-Ihara $L$-functions in \S \ref{l_function}, about voltage assignments in \S \ref{vo_assign}, and about the graph theoretical analogue of Iwasawa's theorem (\ref{Iwasawa}) in \S \ref{Iw_thm}.  In \S \ref{pullb}, we explain the notion of a pullback in the category of graphs and some of its properties.  We reinterpret pullbacks in terms of voltage assignments in \S \ref{pull_b_va}.  We present our main result in \S \ref{ki_main}, which we illustrate through examples in \S \ref{examples}. We present an application to the study of the growth of Iwasawa invariants in noncommutative pro-$\ell$ towers in \S \ref{section 5}, which we illustrate through an example in \S \ref{section 5.2}.

\subsection*{Acknowledgement} The project was initiated when Anwesh Ray was a Simons postdoctoral fellow at the Centre de Recherches Mathématiques in Montreal. He acknowledges support from the CRM-Simons fellowship during this time. The authors thank the anonymous referee for helpful comments.
\section{Preliminaries} \label{preliminaries}

\subsection{Graphs and covers of graphs} \label{gr}
We think of a graph as presented in \cite{Serre:1977} and \cite{Sunada:2013}.  Thus, a graph $X$ consists of a set of vertices $V_{X}$, a set of directed edges $\mathbf{E}_{X}$, an incidence map ${\rm inc}: \mathbf{E}_{X} \rightarrow V_{X} \times V_{X}$ given by $e \mapsto {\rm inc}(e) = (o(e),t(e))$, and an inversion map $\mathbf{E}_{X} \rightarrow \mathbf{E}_{X}$ given by $e \mapsto \bar{e}$ satisfying
\begin{enumerate}
\item $\bar{e} \neq e$,
\item $\bar{\bar{e}} = e$,
\item $o(\bar{e}) = t(e)$ and $t(\bar{e}) = o(e)$,
\end{enumerate}
for all $e \in \mathbf{E}_{X}$.  For $v \in V_{X}$, we let
$$\mathbf{E}_{X,v} = \{e \in \mathbf{E}_{X} \, : \, o(e) = v \}. $$
The set of undirected edges is obtained by identifying $e$ with $\bar{e}$ and will be denoted by $E_{X}$.  Note that we allow loops and parallel edges.  Such graphs are sometimes called multigraphs, but we will refer to them as graphs for simplicity.  \emph{We assume throughout this paper that all of our graphs are finite unless otherwise stated.  This means that both $V_{X}$ and $\mathbf{E}_{X}$ are finite sets.}  Setting $b_{i}(X) := {\rm rank}_{\mathbb{Z}}\, H_{i}(X,\mathbb{Z})$, the Euler characteristic $\chi(X)$ is defined as follows
$$\chi(X) = b_{0}(X) - b_{1}(X).$$ We note that $b_i(X)=0$ for all $i\geq 2$ (cf. \cite[p. 140 l. 7]{HatcherAT}).
If $X$ is connected, then
$$\chi(X) = |V_{X}| - |E_{X}|,$$ which follows from \cite[Section 4.4]{Sunada:2013} or \cite[Theorem 2.44]{HatcherAT}. In this case, we also note that $\chi(X) = 1 - b_{1}(X)$. The valency of a vertex $v \in V_{X}$ is defined as follows
$${\rm val}_{X}(v) = |\mathbf{E}_{X,v}|. $$

\par Henceforth, we choose a labeling of the vertices of $X$, say $V_{X} = \{v_{1},v_{2},\ldots,v_{g} \}$.  The degree (or valency) matrix associated to $X$ is the diagonal matrix $D = (d_{ij})$ satisfying 
\begin{equation} \label{degree_matrix}
d_{ii} = {\rm val}_{X}(v_{i}),
\end{equation} 
for all $i=1,\ldots,g$.  The adjacency matrix $A = (a_{ij})$ of $X$ is the symmetric matrix whose entries are specified as follows
\begin{equation} \label{adj_matrix}
a_{ij} =
\begin{cases}
\text{twice the number of undirected loops at } v_{i}, &\text{ if } i=j;\\
\text{the number of undirected edges between } v_{i} \text{ and } v_{j}, & \text{ otherwise.}
\end{cases}
\end{equation}
The matrix $Q := D - A$ is called the Laplacian matrix and is known to be singular. In fact, it is easy to see that any vector whose entries are all the same is in the null-space of $Q$.

\par In order to introduce the Galois theory of graphs, we first recall the notion of a morphism in the category of graphs. Let $X$ and $Y$ be two graphs.  A morphism of graphs $f:Y \rightarrow X$ consists of two functions $f_{V}:V_{Y} \rightarrow V_{X}$ and $f_{E}:\mathbf{E}_{Y} \rightarrow \mathbf{E}_{X}$ satisfying
\begin{enumerate}
\item $f_{V}(o(e)) = o(f_{E}(e))$,
\item $f_{V}(t(e)) = t(f_{E}(e))$,
\item $\overline{f_{E}(e)} = f_{E}(\bar{e})$, \label{trois}
\end{enumerate}
for all $e \in \mathbf{E}_{X}$.  We will often write $f$ for both $f_{V}$ and $f_{E}$.

\begin{definition} \label{cover}
Let $X$ and $Y$ be two graphs and let $f:Y\rightarrow X$ be a morphism of graphs. If $f$ satisfies the following conditions
\begin{enumerate}
\item $f:V_{Y} \rightarrow V_{X}$ is surjective,
\item for all $w \in V_{Y}$, the restriction $f|_{\mathbf{E}_{Y,w}}$ induces a bijection 
$$f|_{\mathbf{E}_{Y,w}}:\mathbf{E}_{Y,w} \stackrel{\approx}{\rightarrow}{\mathbf{E}_{X,f(w)}}, $$
\end{enumerate}
then $f$ is said to be a cover. We shall often say that $Y/X$ is a cover when the projection map $f$ is understood from the context.  
\end{definition}
\begin{remark} \label{con_cover}
In \cite[Section 5.1]{Sunada:2013}, Sunada requires both $X$ and $Y$ to be connected in the definition of a cover, but we do not make this requirement here.  If $Y$ is assumed to be connected, then so is $X$.  Indeed, if $v_{1}, v_{2} \in V_{X}$, then let $w_{1}, w_{2} \in V_{Y}$ be such that $f(w_{i}) = v_{i}$ for $i=1,2$.  If $Y$ is connected, there exists a path $c$ going from $w_{i}$ to $w_{j}$.  The image $f(c)$ is a path in $X$ going from $v_{i}$ to $v_{j}$ showing that $X$ is connected as well.  On the other hand, if $X$ is assumed to be connected, then $Y$ is not necessarily connected.
\end{remark}
If $Y/X$ is a cover and $Y$ is connected, then there exists a positive integer $d$ such that $|f^{-1}(v)| = d$ for all $v \in V_{X}$.  The integer $d$ is called the degree of $Y/X$ and will be denoted by $[Y:X]$.  If $f:Y \rightarrow X$ is a cover, then we let
$${\rm Aut}_{f}(Y/X) = \{ \sigma \in {\rm Aut}(Y) \, : \, f \circ \sigma = f\}. $$
\begin{definition}
The cover $f:Y \rightarrow X$ is called Galois if the following two conditions are satisfied.
\begin{enumerate}
\item The graph $Y$ is connected (and hence also $X$ by Remark \ref{con_cover}).
\item The group ${\rm Aut}_{f}(Y/X)$ acts transitively on the fiber $f^{-1}(v)$ for all $v \in V_{X}$.
\end{enumerate}
\end{definition}
If $Y/X$ is a Galois cover, then $[Y:X] = |{\rm Aut}_{f}(Y/X)|$, and we write ${\rm Gal}_{f}(Y/X)$, or ${\rm Gal}(Y/X)$ if $f$ is understood, instead of ${\rm Aut}_{f}(Y/X)$.  One has the usual Galois correspondence between subgroups of ${\rm Gal}(Y/X)$ and equivalence classes of intermediate covers of $Y/X$.

\begin{definition} \label{tower}
Let $\ell$ be a rational prime, and let $X$ be a connected graph.  A $\mathbb{Z}_{\ell}$-tower above $X$ consists of a sequence of covers of connected graphs
$$X = X_{0} \leftarrow X_{1} \leftarrow \ldots \leftarrow X_{n} \leftarrow \ldots $$
such that for all positive integer $n$, the cover $X_{n}/X$ obtained from composing the covers $X_{n} \rightarrow X_{n-1} \rightarrow \ldots \rightarrow X_{1} \rightarrow X$ is Galois with Galois group ${\rm Gal}(X_{n}/X)$ isomorphic to $\mathbb{Z}/\ell^{n}\mathbb{Z}$.
\end{definition}

\subsection{Artin-Ihara $L$-functions} \label{l_function}
Let $Y/X$ be an abelian cover of graphs, meaning that $Y/X$ is Galois with abelian group of covering transformations. The theory can be extended to a non-abelian setting, but we restrict ourselves to the abelian setting for simplicity, see \cite{Terras:2011} for further details. Set $G := \op{Gal}(Y/X)$ and denote by $\widehat{G} = {\rm Hom}_{\mathbb{Z}}(G,\mathbb{C}^{\times})$ the group of complex valued characters of $G$.  If $\psi \in \widehat{G}$, then the corresponding Artin-Ihara $L$-function is defined as follows
$$L_{Y/X}(u,\psi) := \prod_{\mathfrak{c}}\left(1 - \psi\left(\frac{Y/X}{\mathfrak{c}} \right)u^{l(\mathfrak{c})} \right)^{-1}. $$
The above product runs over all primes $\mathfrak{c}$ of $X$; $l(\mathfrak{c})$ is the length of $\mathfrak{c}$, and $\left(\frac{Y/X}{\mathfrak{c}} \right) \in G$ denotes the corresponding Frobenius automorphism.  
\begin{remark}For the definition of a prime in a graph, see \cite[Section 2.2]{Terras:2011}, and for the definition of the Frobenius automorphism, see \cite[Section 16]{Terras:2011}.
\end{remark}Recall that we labeled the vertices of $X$ as follows $V_{X} = \{v_{1},v_{2},\ldots,v_{g} \}$.  For each $i=1,\ldots,g$, choose a vertex $w_{i}$ of $Y$ in the fiber of $v_{i}$. For $\sigma \in G$, define the matrix $A(\sigma) = (a_{ij}(\sigma))$ as follows
\begin{equation*}
a_{ij}(\sigma) := 
\begin{cases}
\text{twice the number of undirected loops at the vertex } w_{i}, &\text{if } i=j \text{ and } \sigma =1;\\
\text{the number of undirected edges connecting } w_{i} \text{ to } w_{j}^{\sigma}, &\text{otherwise}.
\end{cases}
\end{equation*}
If $\psi \in \widehat{G}$, one defines the \emph{twisted adjacency matrix} as follows
\begin{equation}\label{twisted}
A_{\psi} := \sum_{\sigma \in G}\psi(\sigma) \cdot A(\sigma). 
\end{equation}
Then the three-term determinant formula \cite[Theorem 18.15]{Terras:2011} asserts that
\begin{equation} \label{3terms}
L_{Y/X}(u,\psi)^{-1} = \left(1 - u^{2}\right)^{-\chi(X)}  \op{det}\left(\op{I} - A_{\psi}u + (D-\op{I})u^{2}\right).
\end{equation}
In the above formula, we recall that $\chi(X)$ is the Euler characteristic of $X$, $D$ is the degree matrix of $X$ (cf. \eqref{degree_matrix}), $A_{\psi}$ the \emph{twisted adjacency matrix} associated to $\psi$ and the cover $Y/X$ (cf. \eqref{twisted}), and $\op{I}$ is the identity matrix. For ease of notation, set
\begin{equation} \label{h_def}
h_{Y/X}(u,\psi) :=  \op{det}\left(\op{I} - A_{\psi}u + (D-\op{I})u^{2}\right) \in \mathbb{Z}[\psi][u],
\end{equation}
where $\mathbb{Z}[\psi]$ denotes the ring of integers in the cyclotomic number field 
$$\Q(\psi) := \mathbb{Q}(\{\psi(\sigma)\, | \, \sigma \in G\}).$$  
One can check that $h_{Y/X}(u,\psi)$ does not depend on the choice of the vertices $w_{i}$.  When $\psi = \psi_{0}$ is the trivial character, then we write $h_{X}(u)$ instead of $h_{Y/X}(u,\psi_{0})$, and $Z_{X}(u)$ instead of $L_{Y/X}(u,\psi_{0})$.  The function $Z_{X}(u)$ is the Ihara zeta function of the graph $X$ introduced in \cite{Ihara:1966} and reinterpreted in the context of graph theory by Sunada in \cite{Sunada:1986} following a suggestion of Serre.  Note that $A_{\psi_{0}} = A$, the usual adjacency matrix of $X$ (cf. \eqref{adj_matrix}). Note that when the cover $Y/X$ is understood, it is suppressed in our notation, and we shall simply write $h_X(u, \psi)$ instead of $h_{Y/X}(u, \psi)$.

It was shown by Stark and Terras (see \cite[Proposition 3]{Stark/Terras:2000}) that the Artin-Ihara $L$-functions satisfy the usual Artin formalism.  It follows that
$$Z_{Y}(u) = Z_{X}(u) \cdot \prod_{\psi \neq \psi_{0}}L_{Y/X}(u,\psi). $$
Furthermore, (\ref{3terms}) and (\ref{RH_graph}) imply that
\begin{equation} \label{prod_dec}
h_{Y}(u) = h_{X}(u) \prod_{\psi \neq \psi_{0}}h_{Y/X}(u,\psi). 
\end{equation}

If $Z$ is an arbitrary connected graph, note that $h_{Z}(1) = 0$ since the Laplacian matrix $Q = D - A$ is singular, and it follows from (\ref{h_def}) that $h_{Z}(1) = \op{det}(Q)$.  In \cite{Hashimoto:1990}, Hashimoto showed that for a connected graph $Z$, one has
\begin{equation} \label{sp_val}
h_{Z}'(1) = -2 \chi(Z) \kappa_{Z}, 
\end{equation}
where $\kappa_{Z}$ is the number of spanning trees of $Z$ (see also \cite{Northshield:1998}).

Coming back to our abelian cover $Y/X$, and assuming that $\chi(X)\neq 0$,  (\ref{RH_graph}), (\ref{prod_dec}), and (\ref{sp_val}) together imply that the following identity holds
\begin{equation} \label{class_number_formula}
|G| \cdot \kappa_{Y} = \kappa_{X} \cdot \prod_{\psi \neq \psi_{0}}h_{Y/X}(1,\psi),
\end{equation}
(cf. \cite[Equation (7)]{Vallieres:2021}).  Very often, we shall make the following assumption on our graphs.
\begin{assumption}\label{euler assumption}
Let $X$ be a connected graph.  We assume that $\chi(X) \neq 0$.
\end{assumption}
\begin{remark}
The special case when $\chi(X) = 0$ can be treated separately. If $X$ is a connected graph such that $\chi(X) = 0$, note that removing all vertices of degree one and their adjacent edge does not change the primes, the Ihara zeta function, and the number of spanning trees. Thus we can assume that ${\rm val}_{X}(v) \ge 2$ for all $v \in V_{X}$.  In this case, $X$ is necessarily the cycle graph on $g$ vertices, which we denote by $C_{g}$.  Now, $\pi_{1}(C_{g},v_{0})$ is a free group on one generator; hence, it is abelian, and 
$$\pi_{1}(C_{g},v_{0}) \simeq H_{1}(C_{g},\mathbb{Z}) \simeq \mathbb{Z}.$$
It follows that up to isomorphism, there is a unique Galois cover $Y/C_{g}$ of degree $n$, it is abelian, with ${\rm Gal}(Y/C_{g}) \simeq \mathbb{Z}/n\mathbb{Z}$.  By (\ref{RH_graph}), we have $Y \simeq C_{gn}$.  Also note that the number of spanning trees of $C_{g}$ is $g$, and thus, the number of spanning trees of $Y$ is $gn$. Thus, the growth patterns in the number of spanning trees in any abelian tower is well understood, due to this explicit description.
\end{remark}

\subsection{Constructions of Galois covers via voltage assignments} \label{vo_assign}
Our main reference for this section is \cite{Gross/Tucker:2001}.  If $X$ is a graph, then we say that a subset $S$ of $\mathbf{E}_{X}$ is an orientation for $X$ if the following two conditions are satisfied:
\begin{enumerate}
\item $S \cap \overline{S} = \varnothing$,
\item $\mathbf{E}_{X} = S \cup \overline{S}$.
\end{enumerate}
In other words, an orientation $S$ is a complete set of representatives for $E_{X}$.  
\begin{remark}
In \cite{Vallieres:2021, mcgownvallieresII, mcgownvallieresIII}, \cite{lei2022non}, \cite{Sage/Vallieres:2022}, and \cite{DLRV}, an orientation was referred to instead as a \emph{section}.
\end{remark}
Given a finite group $G$ and a function $\alpha:S \rightarrow G$ (called a voltage assignment), one can associate a graph $X(G,S,\alpha)$ constructed as follows. Extend $\alpha$ to $\mathbf{E}_X$ by setting $\alpha(\bar{s}) = \alpha(s)^{-1}$ for $s \in S$. Then, the set of vertices of $X(G,S,\alpha)$ is $V_{X} \times G$ and the set of directed edges is $\mathbf{E}_{X} \times G$.  The edge $(e,\sigma)$ connects the vertex $(o(e),\sigma)$ to the vertex $(t(e),\sigma\cdot \alpha(e))$ and the inversion map is given by $\overline{(e,\sigma)} := (\bar{e},\sigma \cdot \alpha(e))$. It is easy to check that $X(G,S,\alpha)$ is a graph in the sense of \S \ref{gr}, which we leave to the reader. In what follows, $v$ (resp. $e$) shall denote a vertex (resp. edge) of $X$, and $\sigma$ shall denote an element of $G$.  The map $p:X(G,S,\alpha) \rightarrow X$ defined via
\begin{equation} \label{proj_map}
p(v,\sigma) := v \text{ and } p(e,\sigma) := e 
\end{equation}
is a cover of graphs.  Given $\tau \in G$, the map $\phi_{\tau}:X(G,S,\alpha) \longrightarrow X(G,S,\alpha)$ defined via
$$\phi_{\tau}(v,\sigma) = (v, \tau \cdot \sigma) \text{ and } \phi_{\tau}(e,\sigma) = (e,\tau \cdot \sigma) $$
satisfies $\phi_{\tau} \in {\rm Aut}_{p}(X(G,S,\alpha)/X)$, and the function 
\begin{equation} \label{iso_au}
G \hookrightarrow {\rm Aut}_{p}(X(G,S,\alpha)/X)
\end{equation}
given by $\tau \mapsto \phi_{\tau}$ is an injective group morphism.  If one assumes that $X(G,S,\alpha)$ is connected, then $X(G,S,\alpha)/X$ is a Galois cover and it follows from the unique lifting theorem \cite[Theorem 5.1]{Sunada:2013} that (\ref{iso_au}) is an isomorphism.

If the base graph $X$ is a bouquet (i.e., has a single vertex), then the graphs $X(G,S,\alpha)$ are precisely the Cayley-Serre graphs, a slight generalization of the usual Cayley graphs.  As explained in \cite[Section 3]{Sage/Vallieres:2022}, given a fixed orientation $S$ of $X$, every abelian cover $Y/X$ with Galois group $G$ is isomorphic as a cover to one of the form $X(G,S,\alpha)/X$ for some function $\alpha:S \rightarrow G$.  

If $G$ and $G_{1}$ are two finite groups and $f:G \rightarrow G_{1}$ is a group morphism, then one obtains a morphism of graphs $f_{*}:X(G,S,\alpha) \rightarrow X(G_{1},S,f \circ \alpha)$ by setting
$$f_{*}(v,\sigma) = (v,f(\sigma)) \text{ and } f_{*}(e,\sigma) = (e,f(\sigma)). $$
If $f$ is surjective, then $f_{*}$ is a cover.  For $\tau \in {\rm ker}(f)$, let $\phi_{\tau}:X(G,S,\alpha) \longrightarrow X(G,S,\alpha)$ be defined via
$$\phi_{\tau}(v,\sigma) = (v, \tau \cdot \sigma) \text{ and } \phi_{\tau}(e,\sigma) = (e,\tau \cdot \sigma). $$
The map 
\begin{equation} \label{ker_iso}
{\rm ker}(f) \hookrightarrow {\rm Aut}_{f_{*}}(X(G,S,\alpha)/X(G_{1},S,f \circ \alpha)),
\end{equation}
defined via $\tau \mapsto \phi_{\tau}$ is a well-defined injective group morphism.  If one assumes furthermore that $X(G,S,\alpha)$ is connected, then $f_{*}:X(G,S,\alpha) \rightarrow X(G_{1},S,f \circ \alpha)$ is a Galois cover and (\ref{ker_iso}) is an isomorphism of groups.

Starting with a function $\alpha:S \rightarrow \mathbb{Z}_{\ell}$, and letting $\alpha_{/n}$ be the reduction of $\alpha$ modulo $\ell^{n}$, it follows from the above discussion that we obtain a $\mathbb{Z}_{\ell}$-tower
\begin{equation} \label{concrete_tower}
X \leftarrow X(\mathbb{Z}/\ell\mathbb{Z},S,\alpha_{/1}) \leftarrow X(\mathbb{Z}/\ell^{2}\mathbb{Z},S,\alpha_{/2})\leftarrow \ldots \leftarrow X(\mathbb{Z}/\ell^{n}\mathbb{Z},S,\alpha_{/n}) \leftarrow \ldots 
\end{equation}
in the sense of Definition \ref{tower}, provided the graphs $X(\mathbb{Z}/\ell^{n}\mathbb{Z},S,\alpha_{/n})$ are connected for all nonnegative integers $n$.  As explained in \cite[Section 2.2]{DLRV}, given an orientation $S$ for $X$, every $\mathbb{Z}_{\ell}$-tower in the sense of Definition \ref{tower} is isomorphic in a suitable sense to a $\mathbb{Z}_{\ell}$-tower as in (\ref{concrete_tower}) for some function $\alpha:S\rightarrow \mathbb{Z}_{\ell}$.
\begin{assumption}\label{main assumption}
Let $\ell$ be a rational prime and let $\alpha: S\rightarrow \Z_\ell$ be a function. We assume that the derived graphs $X(\Z/\ell^n\mathbb{Z}, S, \alpha_{/n})$ are connected for all $n\geq 0$.
\end{assumption}
\begin{remark}
See \cite[Theorem 2.11]{DLRV} for a condition that guarantees the validity of Assumption \ref{main assumption}. We explain the assumption in further detail in \S \ref{conn1} below.
\end{remark}
\subsubsection{Connectedness of graphs obtained from voltage assignments} \label{conn1}
Let $X$ be a connected graph, $S$ an orientation for $X$, $G$ a finite group, and let $\alpha:S \rightarrow G$ be a function which we extend to all of $\mathbf{E}_{X}$ as usual by setting
\begin{equation} \label{sym}
\alpha(\bar{s}) = \alpha(s)^{-1}.    
\end{equation}
If $c = e_{1} e_{2}  \ldots  e_{n}$ is a path in $X$, then we set 
$$\alpha(c) = \alpha(e_{1}) \cdot \ldots \cdot \alpha(e_{n}) \in G. $$
The condition (\ref{sym}) guarantees that if $c_{1}$ and $c_{2}$ are homotopically equivalent, then $\alpha(c_{1}) = \alpha(c_{2})$.  Choosing an arbitrary vertex $v_{0} \in V_{X}$, it follows that $\alpha$ induces a group morphism
\begin{equation} \label{group_mor}
\rho_{\alpha}:\pi_{1}(X,v_{0}) \rightarrow G, 
\end{equation}
defined via $\rho_{\alpha}([\gamma]) = \alpha(\gamma)$.
\begin{theorem} \label{connectedness}
Let $X$ be a connected graph and let $S,G, \alpha$ be as above.  Then the graph $X(G,S,\alpha)$ is connected if and only if the group morphism $\rho_{\alpha}$ defined above in (\ref{group_mor}) is surjective.
\end{theorem}
\begin{proof}
Assume first that $\rho_{\alpha}$ is surjective.  Since $p:X(G,S,\alpha) \rightarrow X$ is a cover, every directed edge $e$ of $X$ can be uniquely lifted to a directed edge of $X(G,S,\alpha)$ once a vertex in the fiber of $o(e)$ has been chosen.  The connectedness of $X$ implies that given any vertices $v_{1}, v_{2} \in V_{X}$ and a choice of $w_{1} \in p^{-1}(v_{1})$, there exists a path in $X(G,S,\alpha)$ starting at $w_{1}$ and ending at a vertex in $p^{-1}(v_{2})$.  To show that $X(G,S,\alpha)$ is connected, it suffices then to show that every vertex in $p^{-1}(v_{0})$ can be connected by a path in $X(G,S,\alpha)$.  Let thus $(v_{0},\sigma), (v_{0},\tau) \in p^{-1}(v_{0})$ be arbitrary.  Since $\rho_{\alpha}$ is assumed to be surjective, there exists a loop $\gamma$ in $X$ based at $v_{0}$ satisfying $\rho_{\alpha}(v_{0}) = \sigma^{-1} \cdot \tau$.  One has
$$\gamma = e_{1}e_{2} \ldots e_{n}, $$
for some $e_{i} \in \mathbf{E}_{X}$.  Then, the path
$$(e_{1},\sigma)(e_{2},\sigma \cdot \alpha(e_{1}))\ldots(e_{n},\sigma \cdot \alpha(e_{1}) \cdot \ldots \cdot \alpha(e_{n-1})) $$
in $X(G,S,\alpha)$ connects $(v_{0},\sigma)$ to $(v_{0},\tau)$ showing the claim.

Conversely, if $X(G,S,\alpha)$ is connected, then there is a path $c$ in $X(G,S,\alpha)$ going from $(v_{0},1_{G})$ to $(v_{0},\sigma)$.  The path $c$ is of the form
$$(e_{1},1_{G})(e_{2}, \alpha(e_{1}))\ldots(e_{n}, \alpha(e_{1}) \cdot \ldots \cdot \alpha(e_{n-1})), $$
for some directed edges $e_{i} \in \mathbf{E}_{X}$ satisfying
\begin{enumerate}
\item $o(e_{1}) = v_{0} = t(e_{n})$, 
\item $t(e_{i}) = o(e_{i+1})$ for $i=1,\ldots,n-1$,
\item $\sigma = \alpha(e_{1}) \cdot \ldots \cdot \alpha(e_{n})$.
\end{enumerate}
Then, the loop $\gamma = p(c)$ satisfies $\rho_{\alpha}([\gamma]) = \sigma$ and since $\sigma \in G$ was chosen arbitrarily, this shows the surjectivity of $\rho_{\alpha}$.
\end{proof}
Theorem \ref{connectedness} can be used to check if $X(G,S,\alpha)$ is connected provided a basis for the free group $\pi_{1}(X,v_{0})$ is known.  Indeed, say $[\gamma_{1}],\ldots,[\gamma_{t}]$ is such a basis, then $X(G,S,\alpha)$ is connected if and only if 
$$G = \langle \rho_{\alpha}([\gamma_{i}]) \, : \, i=1,\ldots,t \rangle. $$
Furthermore, we remind the reader that a basis for $\pi_{1}(X,v_{0})$ can be constructed as follows.  Let $\mathfrak{T}$ be a spanning tree for $X$ and let 
$$S(\mathfrak{T}) = \{s \in S \, | \, \text{The undirected edge corresponding to } s \in \mathfrak{T} \}.$$
For $s \in S \smallsetminus S(\mathfrak{T})$, let $\gamma_{s}$ be the loop based at $v_{0}$ obtained by going through the unique geodesic in $\mathfrak{T}$ going from $v_{0}$ to $o(s)$ followed by $s$ followed by the unique geodesic in $\mathfrak{T}$ going from $t(s)$ to $v_{0}$.  Then the collection
$$\{[\gamma_{s}] \, : \, s \in S \smallsetminus S(\mathfrak{T})\} $$
is a basis for the free group $\pi_{1}(X,v_{0})$.  (See \cite[page 68]{Sunada:2013}.)

\begin{example}
Let $\ell = 2$, $G = \mathbb{Z}/2\mathbb{Z}$ and consider the graph $X$ consisting of four undirected edges connecting two distinct vertices $\{v_{0},v_{1}\}$.  Let $S = \{s_{1},s_{2},s_{3},s_{4} \}$, where $s_{1}, s_{2}$ go from $v_{1}$ to $v_{0}$ and $s_{3}, s_{4}$ go from $v_{0}$ to $v_{1}$.
\[
\begin{tikzpicture}
\draw[fill=black] (0,0) circle (1pt) node[below]{$v_{0}$};
\draw[fill=black] (1,0) circle (1pt) node[below]{$v_{1}$};

\draw (0,0) edge [decoration={markings, mark= at position 0.44 with {\arrow[xscale=-1]{stealth}}},preaction={decorate},bend left=15] (1,0);
\draw (0,0) edge [decoration={markings, mark= at position 0.56 with {\arrow[xscale=1]{stealth}}},preaction={decorate},bend right=15] (1,0);
\draw (0,0) edge [decoration={markings, mark= at position 0.44 with {\arrow[xscale=-1]{stealth}}},preaction={decorate},bend left=40] (1,0);
\draw (0,0) edge [decoration={markings, mark= at position 0.56 with {\arrow[xscale=1]{stealth}}},preaction={decorate},bend right=40] (1,0);
\end{tikzpicture}  
\]

Consider the function $\alpha:S \rightarrow G$ defined via $\alpha(s_{i}) = \bar{1}$ for $i=1,2,3,4$.  Choose the spanning tree $\mathfrak{T}$ consisting of the undirected edge determined by $s_{4}$.  Then,
$$\gamma_{s_{1}} = s_{4}s_{1}, \gamma_{s_{2}} = s_{4}s_{2}, \text{ and } \gamma_{s_{3}} = s_{3} \bar{s}_{4}. $$
One has $\alpha(\gamma_{s_{i}}) = \bar{0}$ for $i=1,2,3$, and thus $X(\mathbb{Z}/2\mathbb{Z},S,\alpha)$ is not connected.
\end{example}

\subsection{The analogue of Iwasawa's theorem} \label{Iw_thm}
Let $X$ be a connected graph satisfying Assumption \ref{euler assumption}, $S$ an orientation for $X$, and $G$ a finite abelian group.  Let also $\alpha:S \rightarrow G$ be a function for which $X(G,S,\alpha)$ is connected.  If we label the vertices of $X$ as $V_{X}=\{v_{1},v_{2},\ldots,v_{g} \}$ and we let $w_{i} = (v_{i},1_{G}) \in V_{X(G,S,\alpha)}$, then \cite[Corollary 5.3]{mcgownvallieresIII} shows that for $\psi \in \widehat{G}$, the corresponding \emph{twisted adjacency matrix} is given by
\begin{equation} \label{twisted_concrete}
A_{\psi} = \left(\sum_{\substack{s \in S \\ {\rm inc}(s) = (v_{i},v_{j})}}\psi(\alpha(s)) + \sum_{\substack{s \in S \\ {\rm inc}(s) = (v_{j},v_{i})}}\psi(-\alpha(s)) \right).
\end{equation}
Recall that $\mu_{\ell^{\infty}}$ is the group of $\ell$-th power roots of unity in $\mathbb{C}^{\times}$.  From now on, we fix an embedding 
$\mathbb{Q}(\mu_{\ell^{\infty}}) \hookrightarrow \overline{\mathbb{Q}}_{\ell} \subseteq \mathbb{C}_{\ell}$,
and we view all characters of a finite abelian $\ell$-group as taking values in $\mathbb{C}_{\ell}^{\times}$ via this embedding.  Let $|\cdot|_{\ell}$ be the usual absolute value on $\mathbb{C}_{\ell}$ normalized so that $|\ell|_{\ell} = \ell^{-1}$. Set 
$D := \{t \in \mathbb{C}_{\ell} \, : \, |t|_{\ell} < 1 \}$, 
and for a continuous group morphism $\psi:\mathbb{Z}_{\ell} \rightarrow \mathbb{C}_{\ell}^{\times}$ of finite order, let $t_\psi\in D$ be the element
$t_{\psi} := \psi(1) - 1$.
Consider now the continuous group morphism 
$\rho:\mathbb{Z}_{\ell} \rightarrow \mathbb{Z}_{\ell}\llbracket T \rrbracket^{\times}$ 
given by $a \mapsto \rho(a) = (1+T)^{a}$ which satisfies
\begin{equation} \label{univ_char}
\rho(a)(t_{\psi}) = \psi(a),
\end{equation}
for all continuous group morphisms $\psi:\mathbb{Z}_{\ell} \rightarrow \mathbb{C}_{\ell}^{\times}$ of finite order.  
\begin{remark}
In \cite{mcgownvallieresIII} and \cite{Sage/Vallieres:2022}, the authors work with the morphism $a \mapsto (1-T)^{a}$ rather than $\rho$ of the present paper, so our convention here is slightly different.
\end{remark}
Given a function $\alpha:S \rightarrow \mathbb{Z}_{\ell}$, we define
$$A_{\rho} = \left(\sum_{\substack{s \in S \\ {\rm inc}(s) = (v_{i},v_{j})}}\rho(\alpha(s)) + \sum_{\substack{s \in S \\ {\rm inc}(s) = (v_{j},v_{i})}}\rho(-\alpha(s)) \right) \in M_{g \times g}(\mathbb{Z}_{\ell}\llbracket T \rrbracket),$$ 
and we let
\begin{equation} \label{char_ser}
f_{X,\alpha}(T) = \op{det}(D - A_{\rho}) \in \mathbb{Z}_{\ell}\llbracket T \rrbracket, 
\end{equation}
where $D$ is the degree matrix of $X$.  Using (\ref{h_def}), (\ref{twisted_concrete}), and (\ref{univ_char}), one has
\begin{equation} \label{pow_ser_ev}
f_{X,\alpha}(t_{\psi}) = h_X(1,\psi), 
\end{equation}
for all continuous characters $\psi:\mathbb{Z}_{\ell} \rightarrow \mathbb{C}_{\ell}^{\times}$ of finite order. At the level $n$ of the tower (\ref{concrete_tower}), formula (\ref{class_number_formula}) becomes
\begin{equation} \label{prod_fo}
\ell^{n} \cdot \kappa_{n} = \kappa_{X} \prod_{\psi \neq \psi_{0}}h_{X}(1,\psi), 
\end{equation}
where the product is over all non-trivial characters $\psi$ of $\mathbb{Z}/\ell^{n}\mathbb{Z} \simeq {\rm Gal}(X(\mathbb{Z}/\ell^{n}\mathbb{Z},S,\alpha_{/n})/X)$.  Recall that to any power series 
$$Q(T) = a_{0} + a_{1}T +a_{2}T^{2}+\ldots \in \mathbb{Z}_{\ell}\llbracket T \rrbracket,$$ 
one associates two invariants, namely
$$\mu(Q(T)) = {\rm min}\{{\rm ord}_{\ell}(a_{i}) \, | \, i \ge 0\},$$
and
$$\lambda(Q(T)) = {\rm min}\{i \ge 0 \, | \, {\rm ord}_{\ell}(a_{i}) = \mu(Q(T))\}.$$
Putting (\ref{pow_ser_ev}), (\ref{prod_fo}) together and using a classical argument about power series in $\mathbb{Z}_{\ell}\llbracket T \rrbracket$ evaluated at points of the form $t_{\psi}$ for continuous characters $\psi:\mathbb{Z}_{\ell} \rightarrow \mathbb{C}_{\ell}^{\times}$ of finite order, one obtains the following result which is \cite[Theorem 6.1]{mcgownvallieresIII}:
\begin{theorem} \label{Iw_anal_graph}
Let $X$ be a connected graph satisfying Assumption \ref{euler assumption}, $\ell$ a rational prime, $S$ an orientation for $X$, and $\alpha:S \rightarrow \mathbb{Z}_{\ell}$ a function for which Assumption \ref{main assumption} is satisfied.  Consider the corresponding $\mathbb{Z}_{\ell}$-tower
\begin{equation} \label{tto}
X \leftarrow X(\mathbb{Z}/\ell\mathbb{Z},S,\alpha_{/1}) \leftarrow X(\mathbb{Z}/\ell^{2}\mathbb{Z},S,\alpha_{/2})\leftarrow \ldots \leftarrow X(\mathbb{Z}/\ell^{n}\mathbb{Z},S,\alpha_{/n}) \leftarrow \ldots,
\end{equation}
and let $\kappa_{n}$ denote the number of spanning trees of $X(\mathbb{Z}/\ell^{n}\mathbb{Z},S,\alpha_{/n})$.  Let also
$$\mu = \mu(f(T)) \text{ and } \lambda = \lambda(f(T)) - 1,$$
where $f(T) = f_{X,\alpha}(T)$ was defined in (\ref{char_ser}) above.  Then, there exist $n_{0},\nu \in \mathbb{Z}$ for which
$${\rm ord}_{\ell}(\kappa_{n}) = \mu \ell^{n} + \lambda n + \nu, $$
when $n\ge n_{0}$.
\end{theorem}
We shall refer to $\mu,\lambda,\nu$ as the Iwasawa invariants associated to the $\mathbb{Z}_{\ell}$-tower (\ref{tto}), and they will be denoted by
$$\mu_{\ell}(X,\alpha),\lambda_{\ell}(X,\alpha), \text{ and } \nu_{\ell}(X,\alpha). $$
We shall often drop the $\ell$ from the notation if it is understood from the context.  Note that
$$\mu_{\ell}(X,\alpha) = \mu(f_{X,\alpha}(T)), \text{ but } \lambda_{\ell}(X,\alpha) = \lambda(f_{X,\alpha}(T)) - 1.$$
\begin{remark}
If $\ell$ is odd, then \cite[Theorem 2.14]{DLRV} shows that $\lambda_{\ell}(X,\alpha)$ is odd.
\end{remark}

\section{Pullbacks} \label{pullb}
In the category of graphs, there is a well-defined notion of pullback as was pointed out for instance in \cite{Stallings:1983}.  If one starts with a diagram
\begin{equation*}
\begin{tikzcd}
Y_{1} \arrow["p_{1}",d] & \\
X & \arrow["p_{2}",l] Y_{2},
\end{tikzcd}
\end{equation*}
where $p_{1}$ and $p_{2}$ are morphisms of graphs, one defines the pullback $Y_{1} \times_{X}Y_{2}$ by defining the set of vertices to be
$$V = \{(w_{1},w_{2}) \in V_{Y_{1}}\times V_{Y_{2}} \, : \, p_{1}(w_{1}) = p_{2}(w_{2}) \}, $$
and the set of directed edges to be
$$\mathbf{E} = \{(e_{1},e_{2}) \in \mathbf{E}_{Y_{1}}\times\mathbf{E}_{Y_{2}} \, : \, p_{1}(e_{1}) = p_{2}(e_{2}) \}. $$
Further, define
$$o(e_{1},e_{2}) = (o(e_{1}),o(e_{2})), t(e_{1},e_{2}) = (t(e_{1}),t(e_{2})) \text{ and } \overline{(e_{1},e_{2})} = (\bar{e}_{1},\bar{e}_{2}). $$
We leave it to the reader to check that $Y_{1}\times_{X}Y_{2}$ is a graph.  For $i=1,2$, the maps $\pi_{i}:Y_{1}\times_{X}Y_{2} \rightarrow Y_{i}$ given by 
$$\pi_{i}(w_{1},w_{2}) = w_{i} \text{ and } \pi_{i}(e_{1},e_{2}) = e_{i} $$
are morphisms of graphs, and the triple $(Y_{1}\times_{X}Y_{2},\pi_{1},\pi_{2})$ satisfies the usual universal property satisfied by a pullback.  We leave the details to the reader.

\begin{remark}
Even if $X$, $Y_{1}$ and $Y_{2}$ are connected graphs, then the pullback $Y_{1} \times_{X} Y_{2}$ is not necessarily connected.  For example, take $X = B_{1}$, the bouquet with one undirected loop, and $Y$ the graph consisting of two vertices $\{v_{1},v_{2}\}$ and two parallel undirected edges between the two vertices.  Say $\mathbf{E}_{X} = \{s,\bar{s}\}$ and $\mathbf{E}_{Y} = \{e_{1},e_{2},\bar{e}_{1},\bar{e}_{2}\}$, where $e_{1}$ goes from $v_{1}$ to $v_{2}$ and $e_{2}$ from $v_{2}$ to $v_{1}$.  Consider also the morphism of graphs $p:Y \rightarrow X$ given by $e_{i} \mapsto s$ and by mapping the two vertices to the unique one of $X$.  If we take $Y_{1} = Y_{2} = Y$ and $p_{1} = p_{2} = p$, then the graph $Y\times_{X}Y$ is disconnected; it is two disjoint copies of $Y$.  The graphs $X$, $Y$ and $Y \times_{X}Y$ are pictured below:
\begin{figure}[h]
\begin{center}
\begin{tikzpicture}
\draw[fill=black] (0,0) circle (1pt);

\path (0,0) edge [loop] (0,0);

\end{tikzpicture}, \, \, \,
\begin{tikzpicture}
\draw[fill=black] (0,0) circle (1pt);
\draw[fill=black] (1,0) circle (1pt);

\path (0,0) edge [bend left=20] (1,0);
\path (0,0) edge [bend right=20] (1,0);

\end{tikzpicture} \, \, \,, \, \,
\text{ and } \, \, \, \, \,
\begin{tikzpicture}
\draw[fill=black] (0,0) circle (1pt);
\draw[fill=black] (1,0) circle (1pt);
\draw[fill=black] (2,0) circle (1pt);
\draw[fill=black] (3,0) circle (1pt);

\path (0,0) edge [bend left=20] (1,0);
\path (0,0) edge [bend right=20] (1,0);
\path (2,0) edge [bend left=20] (3,0);
\path (2,0) edge [bend right=20] (3,0);

\end{tikzpicture}
\end{center}
\end{figure}
\end{remark}

\begin{proposition} \label{cover_preserve}
Consider the diagram
\begin{equation*}
\begin{tikzcd}
Y_{1} \arrow["p_{1}",d] &\arrow["\pi_{1}",l] \arrow["\pi_{2}",d] Y_{1} \times_{X}Y_{2}  \\
X & \arrow["p_{2}",l] Y_{2},
\end{tikzcd}
\end{equation*}
in the category of graphs.  If $p_{2}$ is a cover (as in Definition \ref{cover}), then so is $\pi_{1}:Y_{1}\times_{X}Y_{2} \rightarrow Y_{1}$.
\end{proposition}
\begin{proof}
For the first property, let $w_{1} \in V_{Y_{1}}$.  Since $p_{2}:V_{Y_{2}} \rightarrow V_{X}$ is assumed to be surjective, there exists $w_{2} \in V_{Y_{2}}$ such that $p_{2}(w_{2}) = p_{1}(w_{1})$.  Then $(w_{1},w_{2}) \in V_{Y_{1}\times_{X}Y_{2}}$, and $\pi_{1}(w_{1},w_{2})=w_{1}$.

For the second property, let $w = (w_{1},w_{2}) \in V_{Y_{1}\times_{X}Y_{2}}$.  We show first that
\begin{equation} \label{eq1}
\pi_{1}|_{\mathbf{E}_{Y_{1}\times_{X}Y_{2}},w}:\mathbf{E}_{Y_{1}\times_{X}Y_{2},w} \rightarrow \mathbf{E}_{Y_{1},w_{1}}
\end{equation}
is injective.  Let $(e_{1},e_{2}), (\varepsilon_{1},\varepsilon_{2}) \in \mathbf{E}_{Y_{1}\times_{X}Y_{2},w}$, then in particular $o(e_{2}) = o(\varepsilon_{2}) = w_{2}$.  If $\pi_{1}(e_{1},e_{2}) = \pi_{1}(\varepsilon_{1},\varepsilon_{2})$, then $e_{1} = \varepsilon_{1}$.  It follows that $p_{2}(e_{2}) = p_{1}(e_{1}) = p_{1}(\varepsilon_{1}) = p_{2}(\varepsilon_{2})$ and thus $p_{2}(e_{2}) = p_{2}(\varepsilon_{2})$.  Since $e_{2},\varepsilon_{2} \in \mathbf{E}_{Y_{2},w_{2}}$, the assumption on $p_{2}$ implies that $e_{2} = \varepsilon_{2}$ and this shows that (\ref{eq1}) is injective.  In order to show that (\ref{eq1}) is surjective, let $e_{1} \in \mathbf{E}_{Y_{1},w_{1}}$.  Then $p_{1}(e_{1}) \in \mathbf{E}_{X,p_{1}(w_{1})} = \mathbf{E}_{X,p_{2}(w_{2})}$.  The assumption on $p_{2}$ implies that there exists $e_{2} \in \mathbf{E}_{Y_{2},w_{2}}$ such that $p_{2}(e_{2}) = p_{1}(e_{1})$.  But then $(e_{1},e_{2}) \in \mathbf{E}_{Y_{1}\times_{X}Y_{2},w}$ and $\pi_{1}(e_{1},e_{2}) = e_{1}$.  This ends the proof.
\end{proof}
\begin{proposition} \label{cover_fiber}
Consider the diagram
\begin{equation*}
\begin{tikzcd}
Y \arrow["f",d] &\arrow["\pi_{1}",l] \arrow["f_{1}",d] Y \times_{X}X_{1}  & \\
X & \arrow["p_{1}",l] X_{1} & \arrow["p_{2}",l]X_{2},
\end{tikzcd}
\end{equation*}
and let $p = p_{1} \circ p_{2}$.  Consider also the pullback $(Y \times_{X}X_{2},\pi,f_{2})$ of the diagram
\begin{equation*}
\begin{tikzcd}
Y \arrow["f",d] & \\
X & \arrow["p",l] X_{2}.
\end{tikzcd}
\end{equation*}
By the universal property satisfied by $(Y \times_{X}X_{1},\pi_{1},f_{1})$, we have a unique morphism of graphs $\pi_{2}:Y\times_{X}X_{2} \rightarrow Y \times_{X}X_{1}$ which fits into the following commutative diagram
\begin{equation*}
\begin{tikzcd}
Y \arrow["f",d] &\arrow["\pi_{1}",l] \arrow["f_{1}",d] Y \times_{X}X_{1}  & \arrow["\pi_{2}",l] \arrow["f_{2}",d] Y \times_{X}X_{2} \\
X & \arrow["p_{1}",l] X_{1} & \arrow["p_{2}",l]X_{2},
\end{tikzcd}
\end{equation*}
and for which $\pi = \pi_{1} \circ \pi_{2}$.  If $p_{1}$ and $p_{2}$ are covers, then so is $\pi_{2}$.
\end{proposition}
\begin{proof}
Note that $p$ is a cover, since a composition of covers is a cover.  By Proposition \ref{cover_preserve}, we know that $\pi$ is a cover as well.  The graph morphism $\pi_{2}$ is given by $\pi_{2}(w,v_{2}) = (w,p_{2}(v_{2}))$.  Since $p_{2}$ is surjective on the vertices, it follows directly that $\pi_{2}$ is surjective on the vertices as well.  If $(w,v_{2})\in V_{Y \times_{X}X_{2}}$, and since $\pi_{1}$, $\pi_{2}$ are morphisms of graphs, they induce functions
$$\mathbf{E}_{Y \times_{X}X_{2},(w,v_{2})} \stackrel{\pi_{2}}{\longrightarrow} \mathbf{E}_{Y \times_{X}X_{1},(w,p_{2}(v_{2}))} \stackrel{\pi_{1}}{\longrightarrow} \mathbf{E}_{Y,w},$$
where $\pi_{1}$ is a bijection by Proposition \ref{cover_preserve}.  Since $\pi = \pi_{1} \circ \pi_{2}$ (more precisely the restriction of $\pi$ to the set of directed edges on the left of the displayed line above with target space the set of directed edges one the right of the displayed line above) and $\pi$ is also a bijection, we deduce that $\pi_{2}:  \mathbf{E}_{Y \times_{X}X_{2},(w,v_{2})} \rightarrow \mathbf{E}_{Y \times_{X}X_{1},(w,p_{2}(v_{2}))}$ is also a bijection, and the result follows.

\end{proof}

\begin{proposition} \label{preserve_galois}
Consider the diagram
\begin{equation*}
\begin{tikzcd}
Y_{1} \arrow["p_{1}",d] &\arrow["\pi_{1}",l] \arrow["\pi_{2}",d] Y_{1} \times_{X}Y_{2}  \\
X & \arrow["p_{2}",l] Y_{2},
\end{tikzcd}
\end{equation*}
in the category of graphs.  Assume furthermore that $Y_{1}\times_{X}Y_{2}$ is connected.  If $p_{2}:Y_{2} \rightarrow X$ is a Galois cover, then so is $\pi_{1}:Y_{1}\times_{X}Y_{2} \rightarrow Y_{1}$.  Furthermore, the map ${\rm Gal}(Y_{2}/X) \rightarrow {\rm Gal}(Y_{1}\times_{X}Y_{2}/Y_{1})$ given by $\sigma \mapsto \tilde{\sigma}$, where $\tilde{\sigma}$ is defined via
$$\tilde{\sigma}(w_{1},w_{2}) = (w_{1},\sigma(w_{2})) \text{ and } \tilde{\sigma}(e_{1},e_{2}) = (e_{1},\sigma(e_{2})) $$
is a group isomorphism.
\end{proposition}
\begin{proof}
First, let $\sigma \in {\rm Gal}(Y_{2}/X)$ and define $\tilde{\sigma}$ as above.  We leave it to the reader to check that $\tilde{\sigma} \in {\rm Aut}_{\pi_{1}}(Y_{1}\times_{X}Y_{2}/X)$.  Let $(w_{1},w_{2}),(w_{1},w_{2}') \in \pi_{1}^{-1}(w_{1})$.  We have $p_{2}(w_{2}) = p_{2}(w_{2}')$ and thus $w_{2}, w_{2}'$ are in the same fiber of $p_{2}$.  By assumption, $p_{2}$ is a Galois cover and it follows that there exists $\sigma \in {\rm Gal}(Y_{2}/X)$ such that $\sigma(w_{2}) = w_{2}'$.  But then $\tilde{\sigma}(w_{1},w_{2}) = (w_{1},w_{2}')$ and this shows that ${\rm Aut}_{\pi_{1}}(Y_{1}\times_{X}Y_{2}/Y_{1})$ acts transitively on the fibers of $\pi_{1}$.  Combining with Proposition \ref{cover_preserve}, we deduce that $\pi_{1}$ is also a Galois cover.  We now have a well-defined function 
\begin{equation} \label{gr_iso}
{\rm Gal}(Y_{2}/X) \rightarrow {\rm Gal}(Y_{1}\times_{X}Y_{2}/Y_{1}),
\end{equation}
given by $\sigma \mapsto \tilde{\sigma}$.  Let us show first that (\ref{gr_iso}) is injective.  If $\tilde{\sigma}_{1} = \tilde{\sigma}_{2}$ and if $(w_{1},w_{2}) \in V_{Y_{1}\times_{X}Y_{2}}$, then we have $\sigma_{1}(w_{2}) = \sigma_{2}(w_{2})$.  Since $Y_{2}/X$ is assumed to be Galois, $Y_{2}$ is connected by definition.  The unique lifting theorem (\cite[Theorem 5.1]{Sunada:2013}) implies that $\sigma_{1} = \sigma_{2}$, and this shows the injectivity of (\ref{gr_iso}).  For the surjectivity of (\ref{gr_iso}), let $\tau \in {\rm Gal}(Y_{1}\times_{X}Y_{2}/Y_{1})$ and let $(w_{1},w_{2}) \in V_{Y_{1}\times_{X}Y_{2}}$.  Then $\tau(w_{1},w_{2}) = (w_{1}',w_{2}')$ for some $(w_{1}',w_{2}') \in V_{Y_{1}\times_{X}Y_{2}}$.  Since $\tau$ commutes with $\pi_{1}$, we have $w_{1} = w_{1}'$.  Furthermore, $w_{2}$ and $w_{2}'$ are in the same fiber of $p_{2}$, and since $Y_{2}/X$ is a Galois cover, there exists $\sigma \in {\rm Gal}(Y_{2}/X)$ such that $\sigma(w_{2}) = w_{2}'$.  Therefore, $\tau(w_{1},w_{2}) = \tilde{\sigma}(w_{1},w_{2})$ and the unique lifting theorem implies that $\tau = \tilde{\sigma}$.  The fact that (\ref{gr_iso}) is a group morphism is clear and this ends the proof.
\end{proof}

\begin{proposition} \label{all_disjoint}
Consider the diagram
\begin{equation*}
\begin{tikzcd}
Y_{1} \arrow["p_{1}",d] &\arrow["\pi_{1}",l] \arrow["\pi_{2}",d] Y_{1} \times_{X}Y_{2}  \\
X & \arrow["p_{2}",l] Y_{2},
\end{tikzcd}
\end{equation*}
in the category of graphs.  Assume furthermore that $Y_{1}\times_{X}Y_{2}$ is connected and let $p = p_{i} \circ \pi_{i}, (i=1,2)$.  If both $p_{1}$ and $p_{2}$ are Galois cover, then so is 
$$p:Y_{1} \times_{X} Y_{2} \rightarrow X.$$  
Furthermore, the map ${\rm Gal}(Y_{1}/X) \times {\rm Gal}(Y_{2}/X) \rightarrow {\rm Gal}(Y_{1}\times_{X}Y_{2}/X)$ given by $\sigma  =(\sigma_{1},\sigma_{2}) \mapsto \tilde{\sigma}$, where $\tilde{\sigma}$ is defined via
$$\tilde{\sigma}(w_{1},w_{2}) = (\sigma_{1}(w_{1}),\sigma_{2}(w_{2})) \text{ and } \tilde{\sigma}(e_{1},e_{2}) = (\sigma_{1}(e_{1}),\sigma_{2}(e_{2})) $$
is a group isomorphism.
\end{proposition}
\begin{proof}
First, note that $\pi$ is a cover by using Proposition \ref{cover_preserve} and the fact that compositions of covers is a cover.  We leave it to the reader to check that $\tilde{\sigma}$ is an isomorphism of graphs and that $\tilde{\sigma} \in {\rm Aut}(Y_{1} \times_{X}Y_{2}/X)$.  To show that $Y_{1}\times_{X}Y_{2}/X$ is Galois, we proceed as follows.  Let $w = (w_{1},w_{2})$ and $w' = (w_{1}',w_{2}')$ be two vertices in the same fiber of $p$.  Then, we have $p_{i}(w_{i}) = w_{i}'$ for $i=1,2$.  Since both $Y_{1}/X$ and $Y_{2}/X$ are assumed to be Galois, there exists $\sigma_{i} \in {\rm Gal}(Y_{i}/X)$ satisfying $\sigma_{i}(w_{i}) = w_{i}'$ for $i=1,2$.  Letting $\sigma = (\sigma_{1},\sigma_{2})$, we have $\tilde{\sigma}(w)=w'$ and this shows that $Y_{1}\times_{X}Y_{2}/X$ is Galois.  Using the unique lifting theorem \cite[Theorem 5.1]{Sunada:2013}, one shows using a similar argument as in the proof of Proposition~\ref{preserve_galois} that the map ${\rm Gal}(Y_{1}/X) \times {\rm Gal}(Y_{2}/X) \rightarrow {\rm Gal}(Y_{1}\times_{X}Y_{2}/X)$ given by $\sigma  =(\sigma_{1},\sigma_{2}) \mapsto \tilde{\sigma}$ is a group isomorphism.  The details are left to the reader.
\end{proof}

Let now $X$ and $Y$ be connected graphs, and $f:Y\rightarrow X$ a morphism of graphs.  If we start with a $\mathbb{Z}_{\ell}$-tower
\begin{equation} \label{abs_tow}
X = X_{0} \leftarrow X_{1} \leftarrow \ldots \leftarrow X_{n} \leftarrow \ldots,
\end{equation}
and if we assume that the graphs $Y \times_{X}X_{n}$ are connected for all nonnegative integers $n$, it follows from Propositions \ref{cover_fiber}, \ref{preserve_galois}, and the Galois correspondence that we can pull back the tower (\ref{abs_tow}) along $f$ to obtain a $\mathbb{Z}_{\ell}$-tower
$$Y \leftarrow Y \times_{X}X_{1} \leftarrow Y \times_{X}X_{2} \leftarrow \ldots \leftarrow Y \times_{X}X_{n} \leftarrow \ldots $$
which is above $Y$ now.

\subsection{Pullbacks and voltage assignments} \label{pull_b_va}
If $f:Y\rightarrow X$ is a morphism of graphs for which the induced map $f:\mathbf{E}_{Y} \rightarrow \mathbf{E}_{X}$ is surjective, then any orientation $S$ for $X$ induces a unique orientation $S_{Y}$ for $Y$ by setting
$$S_{Y} = f^{-1}(S). $$
Let now $G$ be a finite abelian group, $\alpha:S \rightarrow G$ a function, and consider the diagram
\begin{equation*}
\begin{tikzcd}
Y \arrow["f",d] & \\
X & \arrow["p",l] X(G,S,\alpha),
\end{tikzcd}
\end{equation*}
where the morphism $p$ is given by (\ref{proj_map}) above.  We have an induced function $\alpha \circ f:S_{Y} \rightarrow G$ and a natural graph morphism
$$\pi_{1}:Y(G,S_{Y},\alpha \circ f) \rightarrow Y, $$
as defined above in (\ref{proj_map}) again.  Furthermore, define $\pi_{2}:Y(G,S_{Y},\alpha \circ f) \rightarrow X(G,S,\alpha)$ via
$$\pi_{2}(w,\sigma) = (f(w),\sigma) \text{ and } \pi_{2}(e,\sigma) = (f(e),\sigma). $$
We leave it to the reader to check that $\pi_{2}$ is a morphism of graphs and that the diagram
\begin{equation*}
\begin{tikzcd}
Y \arrow["f",d] & \arrow["\pi_{1}",l] \arrow["\pi_{2}",d] Y(G,S_{Y},\alpha \circ f) \\
X & \arrow["p",l] X(G,S,\alpha).
\end{tikzcd}
\end{equation*}
commutes.
\begin{proposition} \label{concrete_fiber}
With the notation as above, the triple $(Y(G,S_{Y},\alpha \circ f),\pi_{1},\pi_{2})$ is a pullback for the diagram
\begin{equation*}
\begin{tikzcd}
Y \arrow["f",d] & \\
X & \arrow["p",l] X(G,S,\alpha).
\end{tikzcd}
\end{equation*}
\end{proposition}
\begin{proof}
Let $(Z,\alpha_{1},\alpha_{2})$, where $Z$ is a graph, and let $\alpha_{1}:Z \rightarrow Y$, $\alpha_{2}:Z \rightarrow X(G,S,\alpha)$ be morphisms of graphs satisfying $f \circ \alpha_{1} = p \circ \alpha_{2}$.  Then, for $w \in V_{Z}$ and $e \in \mathbf{E}_{Z}$, one has
$$\alpha_{2}(w) =  (V_{2}(w), G_{2}(w)) \text{ and } \alpha_{2}(e) = (E_{2}(e), G_{2}(e)),$$
for some $V_{2}(w) \in V_{X}$, $E_{2}(e) \in \mathbf{E}_{X}$, and some $G_{2}(w), G_{2}(e) \in G$.  Define 
$$\phi:Z \rightarrow Y(G,S_{Y},\alpha \circ f)$$ 
via
$$w \mapsto (\alpha_{1}(w),G_{2}(w)) \text{ and } e \mapsto (\alpha_{1}(e), G_{2}(e)). $$
We leave it to the reader to check that $\phi$ is the unique morphism of graphs for which $\alpha_{i} = \pi_{i} \circ \phi$ for $i=1,2$, and this concludes the proof.
\end{proof}

Starting now with a function $\alpha:S \rightarrow \mathbb{Z}_{\ell}$ for which Assumption \ref{main assumption} is satisfied, let us assume that Assumption \ref{main assumption} is satisfied as well for $\alpha \circ f:S_{Y}\rightarrow \mathbb{Z}_{\ell} $.  Then, we obtain two $\mathbb{Z}_{\ell}$-towers that fit into the following diagram
\begin{equation} \label{con_pu}
\begin{tikzcd}[sep=1.8em, font=\small]
Y \arrow["f",d] & \arrow[l] \arrow[d] Y(\mathbb{Z}/\ell\mathbb{Z},S_{Y},\alpha \circ f_{/1})& \arrow[l] \ldots&\arrow[l]\arrow[d] Y(\mathbb{Z}/\ell^{n}\mathbb{Z},S_{Y},\alpha \circ f_{/n}) & \arrow[l] \ldots \\
X & \arrow[l] X(\mathbb{Z}/\ell\mathbb{Z},S,\alpha_{/1})  &\arrow[l] \ldots &\arrow[l]X(\mathbb{Z}/\ell^{n}\mathbb{Z},S,\alpha_{/n}) & \arrow[l] \ldots
\end{tikzcd}
\end{equation}
We would like to understand how the Iwasawa invariants for the $\mathbb{Z}_{\ell}$-tower above $X$ are related to the Iwasawa invariants for the $\mathbb{Z}_{\ell}$-tower above $Y$.  
\begin{remark} \label{conn_propagated}
If the map $f:Y \rightarrow X$ in the diagram (\ref{con_pu}) is a cover, and if we assume that $\alpha \circ f:S_{Y}\rightarrow \mathbb{Z}_{\ell}$ satisfies Assumption \ref{main assumption}, then Assumption \ref{main assumption} for $\alpha:S \rightarrow \mathbb{Z}_{\ell}$ is also satisfied by Propositions \ref{cover_preserve}, \ref{concrete_fiber}, and Remark \ref{con_cover}.
\end{remark}

Before stating and proving our main theorem, namely Theorem \ref{main} below, we end this section with the following proposition which will be used in the proof of Theorem \ref{main}.

\begin{proposition} \label{combine_volt}
Let $X$ be a connected graph, and let $G_{1}$, $G_{2}$ be two finite abelian groups.  Let also $S$ be an orientation for $X$.  Suppose we are given two functions $\beta:S \rightarrow G_{2}$, $\alpha:S \rightarrow G_{1}$ and assume that both graphs $X(G_{2},S,\beta)$ and $X(G_{1},S,\alpha)$ are connected.  Let $Y = X(G_{2},S,\beta)$, and consider the diagram
\begin{equation*}
\begin{tikzcd}
Y \arrow["p_{2}",d] & \arrow["\pi_{1}",l]  \arrow["\pi_{2}",d] Y(G_{1},S_{Y},\alpha \circ p_{2}) \\
X & \arrow["p_{1}",l] X(G_{1},S,\alpha),
\end{tikzcd}
\end{equation*}
coming from Proposition \ref{concrete_fiber}.  We set $p = p_{1} \circ \pi_{2} = p_{2} \circ \pi_{1}$.  Consider now the function $\gamma:S \rightarrow G_{2} \times G_{1}$ given by
$$s \mapsto \gamma(s) = (\beta(s),\alpha(s)), $$
and the natural morphism of graphs $\pi:X(G_{2} \times G_{1},S,\gamma) \rightarrow X$.  Then, there is an isomorphism of graphs
$$\phi: X(G_{2} \times G_{1},S,\gamma) \rightarrow Y(G_{1},S_{Y},\alpha \circ p_{2})$$
satisfying $p \circ \phi = \pi$.
\end{proposition}
\begin{proof}
The vertex and directed edge sets of $X(G_{2}\times G_{1},S,\gamma)$ are
$$V_{X} \times (G_{2} \times G_{1}) \text{ and }  \mathbf{E}_{X} \times (G_{2} \times G_{1}), $$
whereas the vertex and directed edge sets of $Y(G_{1},S_{Y},\alpha \circ p_{2})$ are
$$ (V_{X} \times G_{2}) \times G_{1} \text{ and } (\mathbf{E}_{X} \times G_{2}) \times G_{1}.$$
It is then obvious what the isomorphism of graphs 
$$\phi:X(G_{2} \times G_{1},S,\gamma) \rightarrow Y(G_{1},S_{Y},\alpha \circ p_{2})$$ 
is, and we leave the details to the reader.
\end{proof}

\section{The analogue of Kida's formula} \label{ki_main}
\subsection{Main result on Kida's formula} In this section, we state and prove the main result.
\begin{theorem} \label{main}
Let $\ell$ be a rational prime, $X$ a connected graph satisfying Assumption \ref{euler assumption}, $S$ an orientation for $X$, and $\alpha:S \rightarrow \mathbb{Z}_{\ell}$ a function.  Let $p:Y \rightarrow X$ be a Galois cover for which $G= {\rm Gal}(Y/X)$ is a finite $\ell$-group. Consider the orientation $S_{Y}$ of $Y$, the induced function $\alpha \circ p:S_{Y} \rightarrow \mathbb{Z}_{\ell}$, and assume that $\alpha \circ p$ satisfies Assumption \ref{main assumption}.  By Remark \ref{conn_propagated}, $\alpha:S \rightarrow \mathbb{Z}_{\ell}$ also satisfies Assumption \ref{main assumption}, and we have two $\mathbb{Z}_{\ell}$-towers that fit into the diagram  
\begin{equation*}
\begin{tikzcd}
Y \arrow["p",d] & \arrow[l] \arrow[d] Y(\mathbb{Z}/\ell\mathbb{Z},S_{Y},\alpha \circ p_{/1})& \arrow[l] \ldots&\arrow[l]\arrow[d] Y(\mathbb{Z}/\ell^{n}\mathbb{Z},S_{Y},\alpha \circ p_{/n}) & \arrow[l] \ldots \\
X & \arrow[l] X(\mathbb{Z}/\ell\mathbb{Z},S,\alpha_{/1})  &\arrow[l] \ldots &\arrow[l]X(\mathbb{Z}/\ell^{n}\mathbb{Z},S,\alpha_{/n}) & \arrow[l] \ldots
\end{tikzcd}
\end{equation*}
Then, $\mu(X,\alpha) = 0$ if and only if $\mu(Y,\alpha \circ p)=0$, in which case
\begin{equation}\label{main formula}\lambda(Y,\alpha \circ p) + 1 = [Y:X] \cdot (\lambda(X,\alpha) + 1).\end{equation}
\end{theorem}

\begin{remark}
Note that in the above formula, the degree "$[Y:X]$" plays a role, while in \eqref{kida}, the analogue is the degree "$[E_\infty:F_\infty]$". If one specializes \eqref{kida} to the case when $E$ and $F_\infty$ are linearly disjoint over $F$, then of course, $[E_\infty:F_\infty]=[E:F]$. The Assumption \ref{main assumption} on the pullback tower over $Y$, may be viewed as a sort of "disjointedness" assumption. For instance, if $p:Y\rightarrow X$ is the map $X_n\rightarrow X$, then, the tower over $X_n$ obtained by pullback fails to be connected.
\end{remark}

\begin{proof}
Every $\ell$-group $G$ has a non-trivial center $Z(G)$, and by Cauchy's theorem, there exists a subgroup $H \le Z(G)$ such that $|H| = \ell$.  Since $H \subseteq Z(G)$, $H$ is a normal subgroup of $G$.  By the Galois correspondence, there exists an intermediate cover $(Z,q_{1},q_{2})$ corresponding to the subgroup $H$ fitting into a commutative diagram
\begin{equation*}
\begin{tikzcd}
Y \arrow["p",dd] \arrow[rd,"q_{1}"] & \\
& Z \arrow[ld,"q_{2}"]\\
X & 
\end{tikzcd}
\end{equation*}
and the normality of $H$ implies that both $q_{1}$ and $q_{2}$ are Galois covers.  The cover $Y/Z$ has Galois group $H$, and $Z/X$ is again a Galois cover whose Galois group $G/H$ is an $\ell$-group.  Note that $\alpha \circ q_{2}:S_{Z} \rightarrow \mathbb{Z}_{\ell}$ satisfies Assumption \ref{main assumption} by Remark \ref{conn_propagated}.  Repeating the argument above, we can construct a finite tower of covers of connected graphs
$$X = Z_{0} \leftarrow Z_{1} \leftarrow \ldots \leftarrow Z_{t} = Y, $$
for which $Z_{i}/Z_{i-1}$ is Galois with Galois group isomorphic to $\mathbb{Z}/\ell\mathbb{Z}$ ($i=1,\ldots,t$).  Furthermore, $[Y:X] = \ell^{t} = |{\rm Gal}(Y/X)|$.  Therefore, it suffices to show the theorem in the situation where $Y/X$ is a cyclic cover of order $\ell$, and we assume so from now on.

As explained in \cite[Section 3]{Sage/Vallieres:2022}, there exists a function $\beta:S \rightarrow G$ for which the cover $X(G,S,\beta)/X$ is isomorphic to the cover $Y/X$, so we can replace $Y$ by $X(G,S,\beta)$.  From now on, let $\Gamma_{n} = {\rm Gal}(Y(\mathbb{Z}/\ell^{n}\mathbb{Z},S_{Y},\alpha \circ p_{/n})/Y) \simeq \mathbb{Z}/\ell^{n}\mathbb{Z}$.  Note that by Proposition \ref{all_disjoint}, the cover $Y(\mathbb{Z}/\ell^{n}\mathbb{Z},S_{Y},\alpha \circ p_{/n})/X$ is an abelian cover with Galois group isomorphic to the cartesian product $G \times \Gamma_{n}$, and by Proposition \ref{combine_volt}, we can replace the graphs $Y(\mathbb{Z}/\ell^{n}\mathbb{Z},S_{Y},\alpha \circ p_{/n})$ by the graphs $X(G \times \Gamma_{n},S,\gamma_{/n})$, where $\gamma:S \rightarrow G \times \mathbb{Z}_{\ell}$ is given by $\gamma(s) = (\beta(s),\alpha(s))$, and $\gamma_{/n}$ represents the reduction modulo $\ell^{n}$ of $\gamma$ in the second component.

Let $\mathfrak{o}$ be the ring of integers of the local field $K = \mathbb{Q}_{\ell}(\mu_{\ell})$.  If $\psi \in \widehat{G}$, then we define
$$A_{\psi,\rho} =  \left(\sum_{\substack{s \in S \\ {\rm inc}(s) = (v_{i},v_{j})}}\psi(\beta(s))\rho(\alpha(s)) + \sum_{\substack{s \in S \\ {\rm inc}(s) = (v_{j},v_{i})}}\psi(-\beta(s))\rho(-\alpha(s)) \right) \in M_{g \times g}(\mathfrak{o}\llbracket T \rrbracket),$$
and set
$$f_{\psi}(T) = \op{det}(D - A_{\psi,\rho}) \in \mathfrak{o}\llbracket T \rrbracket, $$
where $D$ is as usual the degree matrix of $X$.  Note that
$$f_{\psi_{0}}(T) = f_{X,\alpha}(T) \in \mathbb{Z}_{\ell}\llbracket T \rrbracket. $$  
The characters of $G \times \Gamma_{n}$ are of the form $\psi \otimes \phi$, where $\psi \in \widehat{G}$ and $\phi \in \widehat{\Gamma}_{n}$.  Now, view $\phi$ as a character on $1 \times \Gamma_{n} \le G \times \Gamma_{n}$.  We claim that  
\begin{equation} \label{rep_iso}
{\rm Ind}_{1 \times \Gamma_{n}}^{G \times \Gamma_{n}}(\phi) \simeq \bigoplus_{\psi \in \widehat{G}} \psi \otimes \phi.
\end{equation}
To prove the claim, let $W$ be the $\mathbb{C}[\Gamma_{n}]$-module corresponding to the character $\phi$.  Then
\begin{equation*}
\begin{aligned}
\mathbb{C}[G \times \Gamma_{n} ] \otimes_{\mathbb{C}[\Gamma_{n}]}W &\simeq \mathbb{C}[G] \otimes_{\mathbb{C}} \mathbb{C}[ \Gamma_{n}] \otimes_{\mathbb{C}[ \Gamma_{n}]} W \\
&\simeq \mathbb{C}[G] \otimes_{\mathbb{C}}W \\
&\simeq \Big(\bigoplus_{\psi \in \widehat{G}} \mathbb{C}e_{\psi}\Big) \otimes_{\mathbb{C}} W \\
&\simeq \bigoplus_{\psi \in \widehat{G}} \mathbb{C}e_{\psi} \otimes_{\mathbb{C}}W,
\end{aligned}
\end{equation*}
and this shows the claim (\ref{rep_iso}) above.  Using the Artin formalism satisfied by the Artin-Ihara $L$-functions \cite[Proposition 3]{Stark/Terras:2000} and (\ref{RH_graph}), we get
\begin{equation} \label{useful_prod}
h_{Y}(u,\phi) = \prod_{\psi \in \widehat{G}}h_{Y/X}(u,\psi \otimes \phi),
\end{equation}
for all characters $\phi \in \widehat{\Gamma}_{n}$ and for all $n \ge 0$.  

Combining (\ref{h_def}), (\ref{twisted_concrete}), and (\ref{univ_char}) give
\begin{equation} \label{special_value_ps}
f_{\psi}(t_{\phi}) = h_{X}(1,\psi \otimes \phi),
\end{equation}
for all characters $\phi \in \widehat{\Gamma}_{n}$, for all $\psi \in \widehat{G}$, and for all $n \ge 0$.  Combining further with (\ref{useful_prod}), we obtain
$$f_{Y,\alpha \circ p}(t_{\phi}) = \prod_{\psi \in \widehat{G}} f_{\psi}(t_{\phi}) $$
for all characters $\phi \in \widehat{\Gamma}_{n}$ and for all $n \ge 0$.  Since a power series in $\mathfrak{o}\llbracket T \rrbracket $ has finitely many zeros in the unit disk $D = \{t \in \mathbb{C}_{\ell} \, : \, |t|_{\ell}< 1\}$, the equalities (\ref{special_value_ps}) above imply
\begin{equation} \label{essential_eq}
f_{Y,\alpha \circ p}(T) = \prod_{\psi \in \widehat{G}}f_{\psi}(T).
\end{equation}

Let ${\rm ord}_{\ell}$ be the valuation on $\mathbb{C}_{\ell}$ normalized so that ${\rm ord}_{\ell}(\ell)=1$.  One defines the Iwasawa invariant of a non-zero power series
$$Q(T) = a_{0} + a_{1}T + a_{2}T^{2} + \ldots \in \mathfrak{o}\llbracket T \rrbracket $$
similarly as before, namely
$$\mu(Q(T)) = {\rm min}\{{\rm ord}_{\ell}(a_{i}) \, | \, i \ge 0\},$$
and
$$\lambda(Q(T)) = {\rm min}\{i \ge 0 \, | \, {\rm ord}_{\ell}(a_{i}) = \mu(Q(T))\}.$$
Note that $\mu(Q(T))$ is a nonnegative rational number and $\lambda(Q(T))$ is an integer.  If $Q_{1}(T), Q_{2}(T) \in \mathfrak{o}\llbracket T \rrbracket$, then
\begin{equation}\label{mul_inv}
\mu(Q_{1} \cdot Q_{2}) = \mu(Q_{1}) + \mu(Q_{2}) \text{ and } \lambda(Q_{1} \cdot Q_{2}) = \lambda(Q_{1}) + \lambda(Q_{2}).
\end{equation}

If $\psi$ is a non-trivial character of $G$, and $\mathcal{L}$ is the unique maximal ideal of $\mathfrak{o}$, then 
\begin{equation} \label{cong}
A_{\psi,\rho} \equiv A_{\rho} \pmod{M_{g \times g}(\mathcal{L}\llbracket T \rrbracket)}. 
\end{equation}
Let us denote the reduction map $\mathfrak{o}\llbracket T \rrbracket \longrightarrow \mathbb{F}_{\ell}\llbracket T \rrbracket$ by $f(T) \mapsto \bar{f}(T)$.  Equation (\ref{cong}) implies
$$\bar{A}_{\psi,\rho} = \bar{A}_{\rho} \text{ in } M_{g \times g}(\mathbb{F}_{\ell}\llbracket T \rrbracket), $$
and it follows that
\begin{equation} \label{eq_mod}
\bar{f}_{\psi}(T) = \bar{f}_{X,\alpha}(T) \text{ in } \mathbb{F}_{\ell}\llbracket T \rrbracket, 
\end{equation}
for all non-trivial characters $\psi$ of $G$.  We deduce that if $\mu(f_{X,\alpha}(T)) = 0$, then $\mu(f_{\psi}(T)) = 0$ for all non-trivial characters $\psi$ of $G$.  Combined with (\ref{essential_eq}) and (\ref{mul_inv}) above, one gets that if $\mu(f_{X,\alpha}(T)) = 0$, then $\mu(f_{Y,\alpha \circ p}(T)) = 0$.  Since the $\mu$ invariant is nonnegative, the converse is clear.  Under the assumption that $\mu(X,\alpha) = 0$, it follows from (\ref{eq_mod}) that
$$\lambda(f_{\psi}(T)) = \lambda(f_{X,\alpha}(T)), $$
for all non-trivial characters $\psi$ of $G$.  Therefore, (\ref{essential_eq}) and (\ref{mul_inv}) again imply the desired result.

\end{proof}

\subsection{Examples} \label{examples}

\begin{example}
Consider the bouquet with three loops $X = B_{3}$, and pick an orientation $S = \{s_{1},s_{2},s_{3}\}$.  Let $\ell=3$, and consider $\alpha:S \rightarrow \mathbb{Z} \subseteq \mathbb{Z}_{3}$ given by 
$$s_{1} \mapsto 1, s_{2} \mapsto 4, \text{ and } s_{3} \mapsto 20.$$
This is Example $3$ on page $451$ of \cite{Vallieres:2021}.  It was shown in \emph{loc. cit.} that 
$${\rm ord}_{3}(\kappa_{n}) = 5n - 2, $$
for all $n \ge 1$.  Thus,
$$\mu(B_{3},\alpha) = 0 \text{ and } \lambda(B_{3},\alpha) = 5. $$
The $\mathbb{Z}_{3}$-tower above $X = B_{3}$ is the following:
\begin{equation*}
\begin{tikzpicture}[baseline={([yshift=-1.7ex] current bounding box.center)}]
\node[draw=none,minimum size=3cm,regular polygon,regular polygon sides=1] (a) {};
\foreach \x in {1}
  \fill (a.corner \x) circle[radius=0.7pt];
\draw (a.corner 1) to [in=50,out=130,loop] (a.corner 1);
\draw (a.corner 1) to [in=50,out=130,distance = 0.8cm,loop] (a.corner 1);
\draw (a.corner 1) to [in=50,out=130,distance = 0.5cm,loop] (a.corner 1);
\end{tikzpicture}
\longleftarrow \, \, \,
\begin{tikzpicture}[baseline={([yshift=-0.6ex] current bounding box.center)}]
\node[draw=none,minimum size=2cm,regular polygon,regular polygon sides=3] (a) {};

\foreach \x in {1,2,3}
  \fill (a.corner \x) circle[radius=0.7pt];

\path (a.corner 1) edge [bend left=20] (a.corner 2);
\path (a.corner 1) edge [bend right=20] (a.corner 2);
\path (a.corner 2) edge [bend left=20] (a.corner 3);
\path (a.corner 2) edge [bend right=20] (a.corner 3);
\path (a.corner 3) edge [bend left=20] (a.corner 1);
\path (a.corner 3) edge [bend right=20] (a.corner 1);

\path (a.corner 1) edge  (a.corner 2);
\path (a.corner 2) edge  (a.corner 3);
\path (a.corner 3) edge  (a.corner 1);

\end{tikzpicture}
\longleftarrow \, \,
\begin{tikzpicture}[baseline={([yshift=-0.6ex] current bounding box.center)}]
\node[draw=none,minimum size=2cm,regular polygon,regular polygon sides=9] (a) {};

\foreach \x in {1,2,...,9}
  \fill (a.corner \x) circle[radius=0.7pt];
  
\foreach \y\z in {1/2,2/3,3/4,4/5,5/6,6/7,7/8,8/9,9/1}
  \path (a.corner \y) edge (a.corner \z);
  
\foreach \y\z in {1/3,2/4,3/5,4/6,5/7,6/8,7/9,8/1,9/2}
  \path (a.corner \y) edge (a.corner \z); 
  
\foreach \y\z in {1/5,2/6,3/7,4/8,5/9,6/1,7/2,8/3,9/4}
  \path (a.corner \y) edge (a.corner \z);

\end{tikzpicture}
\longleftarrow
\begin{tikzpicture}[baseline={([yshift=-0.6ex] current bounding box.center)}]
\node[draw=none,minimum size=2cm,regular polygon,regular polygon sides=27] (a) {};

\foreach \x in {1,2,...,27}
  \fill (a.corner \x) circle[radius=0.7pt];
  
\foreach \y\z in {1/2,2/3,3/4,4/5,5/6,6/7,7/8,8/9,9/10,10/11,11/12,12/13,13/14,14/15,15/16,16/17,17/18,18/19,19/20,20/21,21/22,22/23,23/24,24/25,25/26,26/27,27/1}
  \path (a.corner \y) edge (a.corner \z);
  
\foreach \y\z in {1/5,2/6,3/7,4/8,5/9,6/10,7/11,8/12,9/13,10/14,11/15,12/16,13/17,14/18,15/19,16/20,17/21,18/22,19/23,20/24,21/25,22/26,23/27,24/1,25/2,26/3,27/4}
  \path (a.corner \y) edge (a.corner \z); 
  
\foreach \y\z in {1/21,2/22,3/23,4/24,5/25,6/26,7/27,8/1,9/2,10/3,11/4,12/5,13/6,14/7,15/8,16/9,17/10,18/11,19/12,20/13,21/14,22/15,23/16,24/17,25/18,26/19,27/20}
  \path (a.corner \y) edge (a.corner \z);

\end{tikzpicture}
\, \, \longleftarrow \ldots
\end{equation*}
Now let $\beta:S \rightarrow \mathbb{Z}/3\mathbb{Z}\times \mathbb{Z}/ 3\mathbb{Z}$ be defined via
$$s_{1}, s_{3} \mapsto (\bar{1},\bar{0}) \text{ and } s_{2} \mapsto (\bar{0},\bar{1}).$$
Since $\beta(S)$ generates $\mathbb{Z}/3\mathbb{Z} \times \mathbb{Z}/3\mathbb{Z}$, the graph $ Y = X(\mathbb{Z}/3\mathbb{Z}\times \mathbb{Z}/3\mathbb{Z},S,\beta)$ is connected and thus the cover $p:Y \rightarrow X$ is Galois with Galois group isomorphic to $\mathbb{Z}/3\mathbb{Z} \times \mathbb{Z}/3\mathbb{Z}$: 

\begin{equation*}
\begin{tikzpicture}[baseline={([yshift=-1.7ex] current bounding box.center)}]
\node[draw=none,minimum size=3cm,regular polygon,regular polygon sides=1] (a) {};
\foreach \x in {1}
  \fill (a.corner \x) circle[radius=0.7pt];
\draw (a.corner 1) to [in=50,out=130,loop] (a.corner 1);
\draw (a.corner 1) to [in=50,out=130,distance = 0.8cm,loop] (a.corner 1);
\draw (a.corner 1) to [in=50,out=130,distance = 0.5cm,loop] (a.corner 1);
\end{tikzpicture}
\stackrel{p}{\longleftarrow} \, \, \,
\begin{tikzpicture}[baseline={([yshift=-0.6ex] current bounding box.center)}]
\node[draw=none,minimum size=2cm,regular polygon,regular polygon sides=9] (a) {};

\foreach \x in {1,2,...,9}
  \fill (a.corner \x) circle[radius=0.7pt];
  
\path (a.corner 1) edge [bend left=12] (a.corner 2);
\path (a.corner 1) edge [bend right=12] (a.corner 2);

\path (a.corner 2) edge [bend left=12] (a.corner 3);
\path (a.corner 2) edge [bend right=12] (a.corner 3);

\path (a.corner 1) edge [bend left=8] (a.corner 3);
\path (a.corner 1) edge [bend right=8] (a.corner 3);

\path (a.corner 4) edge [bend left=12] (a.corner 5);
\path (a.corner 4) edge [bend right=12] (a.corner 5);

\path (a.corner 5) edge [bend left=12] (a.corner 6);
\path (a.corner 5) edge [bend right=12] (a.corner 6);

\path (a.corner 4) edge [bend left=8] (a.corner 6);
\path (a.corner 4) edge [bend right=8] (a.corner 6);

\path (a.corner 7) edge [bend left=12] (a.corner 8);
\path (a.corner 7) edge [bend right=12] (a.corner 8);

\path (a.corner 8) edge [bend left=12] (a.corner 9);
\path (a.corner 8) edge [bend right=12] (a.corner 9);

\path (a.corner 7) edge [bend left=8] (a.corner 9);
\path (a.corner 7) edge [bend right=8] (a.corner 9);

\path (a.corner 1) edge  (a.corner 4);
\path (a.corner 2) edge  (a.corner 5);
\path (a.corner 3) edge  (a.corner 6);
\path (a.corner 4) edge  (a.corner 7);
\path (a.corner 5) edge  (a.corner 8);
\path (a.corner 6) edge  (a.corner 9);
\path (a.corner 7) edge  (a.corner 1);
\path (a.corner 8) edge  (a.corner 2);
\path (a.corner 9) edge  (a.corner 3);

\end{tikzpicture}
\end{equation*}
Consider now $S_{Y} = p^{-1}(S)$ and $\alpha \circ p: S_{Y} \rightarrow \mathbb{Z} \subseteq \mathbb{Z}_{3}$.  To show that $\alpha \circ p: S_{Y} \rightarrow \mathbb{Z}_{3}$ satisfies Assumption \ref{main assumption}, we can use Proposition \ref{combine_volt} and show instead that the graphs $X(\mathbb{Z}/3\mathbb{Z} \times \mathbb{Z}/3\mathbb{Z} \times \mathbb{Z}/3^{n}\mathbb{Z},S,\gamma_{/n})$ are connected for all $n \ge 1$.  But a simple calculation shows that $\gamma_{/n}(S)$ generates $\mathbb{Z}/3\mathbb{Z} \times \mathbb{Z}/3\mathbb{Z} \times \mathbb{Z}/3^{n}\mathbb{Z}$ for all $n \ge 1$, and this shows that Assumption \ref{main assumption} is satisfied for $\alpha \circ p: S_{Y} \rightarrow \mathbb{Z}_{3}$.  We obtain the $\mathbb{Z}_{3}$-tower above $Y$:
\begin{equation*}
\begin{tikzpicture}[baseline={([yshift=-0.6ex] current bounding box.center)}]
\node[draw=none,minimum size=2cm,regular polygon,regular polygon sides=9] (a) {};

\foreach \x in {1,2,...,9}
  \fill (a.corner \x) circle[radius=0.7pt];
  
\path (a.corner 1) edge [bend left=12] (a.corner 2);
\path (a.corner 1) edge [bend right=12] (a.corner 2);

\path (a.corner 2) edge [bend left=12] (a.corner 3);
\path (a.corner 2) edge [bend right=12] (a.corner 3);

\path (a.corner 1) edge [bend left=8] (a.corner 3);
\path (a.corner 1) edge [bend right=8] (a.corner 3);

\path (a.corner 4) edge [bend left=12] (a.corner 5);
\path (a.corner 4) edge [bend right=12] (a.corner 5);

\path (a.corner 5) edge [bend left=12] (a.corner 6);
\path (a.corner 5) edge [bend right=12] (a.corner 6);

\path (a.corner 4) edge [bend left=8] (a.corner 6);
\path (a.corner 4) edge [bend right=8] (a.corner 6);

\path (a.corner 7) edge [bend left=12] (a.corner 8);
\path (a.corner 7) edge [bend right=12] (a.corner 8);

\path (a.corner 8) edge [bend left=12] (a.corner 9);
\path (a.corner 8) edge [bend right=12] (a.corner 9);

\path (a.corner 7) edge [bend left=8] (a.corner 9);
\path (a.corner 7) edge [bend right=8] (a.corner 9);

\path (a.corner 1) edge  (a.corner 4);
\path (a.corner 2) edge  (a.corner 5);
\path (a.corner 3) edge  (a.corner 6);
\path (a.corner 4) edge  (a.corner 7);
\path (a.corner 5) edge  (a.corner 8);
\path (a.corner 6) edge  (a.corner 9);
\path (a.corner 7) edge  (a.corner 1);
\path (a.corner 8) edge  (a.corner 2);
\path (a.corner 9) edge  (a.corner 3);

\end{tikzpicture}
\, \, \, \, \,  \longleftarrow 
\parbox{.20\linewidth}{
    \centering
        \includegraphics[scale=0.25]{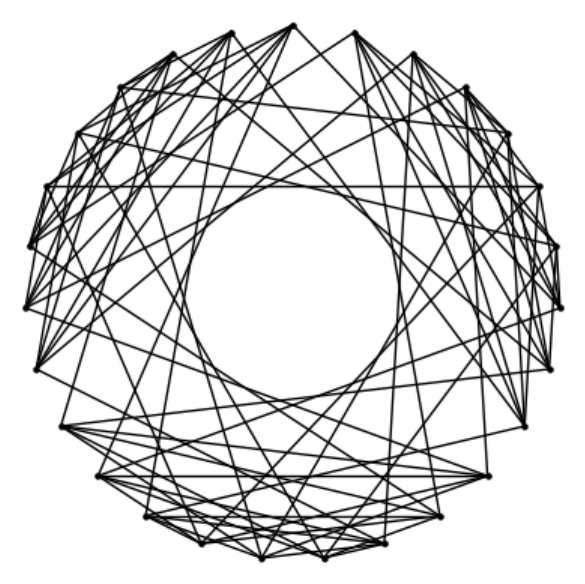}
}
\longleftarrow
\parbox{.20\linewidth}{
    \centering
        \includegraphics[scale=0.25]{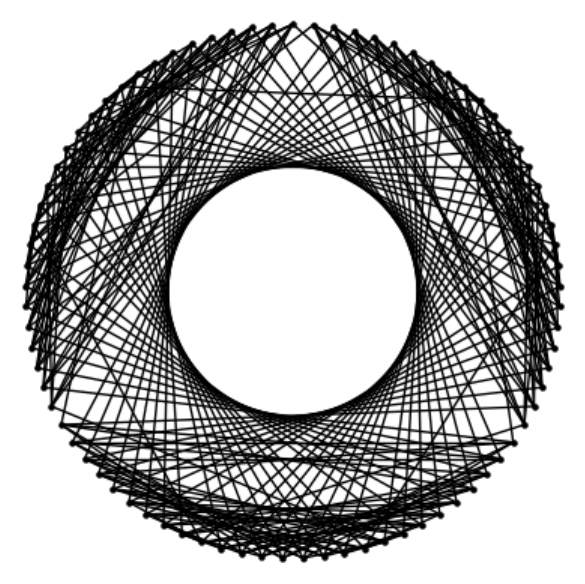}
}
\longleftarrow \, \, \ldots
\end{equation*}
Using \eqref{char_ser}, we calculate the characteristic series 
$$f_{Y,\alpha \circ p}(T) = - 886443588T^{2} + 886443588T^{3} - 7697155248T^{4} + \ldots  \in \mathbb{Z}\llbracket T \rrbracket \subseteq \mathbb{Z}_{3}\llbracket T \rrbracket,$$ 
for which the first coefficient that is not divisible by $3$ is the coefficient of $T^{54}$.  Thus, we obtain
$$\mu(Y,\alpha \circ p)=0 \text{ and } \lambda(Y,\alpha \circ p) = 53. $$
Using SageMath \cite{SAGE}, we calculate
$$\kappa_{0} =2^{2} \cdot 3^{10}, \kappa_{1} = 2^{12} \cdot 3^{31}, \kappa_{2} = 2^{24} \cdot 3^{74}\cdot 17^{6} \cdot 19^6,$$
and 
$$\kappa_{3} = 2^{24} \cdot 3^{127} \cdot 17^{6} \cdot 19^{6} \cdot 102761^{6} \cdot 134243^{6} \cdot 176417^{6}. $$
We have
$${\rm ord}_{3}(\kappa_{n}) = 53n - 32,$$
for $n \ge 2$.  Here, we have $[Y:X] = 9$ and the relation
$$\lambda(Y,\alpha \circ p) + 1 = [Y:X] \cdot (\lambda(B_{3},\alpha) + 1), $$
as expected by Theorem \ref{main}.
\end{example}

\begin{example}
Consider again the bouquet with three loops $X=B_{3}$, and pick an orientation $S = \{ s_{1},s_{2},s_{3}\}$.  Let now $\ell=2$, and consider $\alpha:S \rightarrow \mathbb{Z} \subseteq \mathbb{Z}_{2}$ given by
$$s_{1},s_{2},s_{3} \mapsto 1. $$
One obtains a $\mathbb{Z}_{2}$-tower
\begin{equation*}
\begin{tikzpicture}[baseline={([yshift=-1.7ex] current bounding box.center)}]
\node[draw=none,minimum size=3cm,regular polygon,regular polygon sides=1] (a) {};
\foreach \x in {1}
  \fill (a.corner \x) circle[radius=0.7pt];
\draw (a.corner 1) to [in=50,out=130,loop] (a.corner 1);
\draw (a.corner 1) to [in=50,out=130,distance = 0.8cm,loop] (a.corner 1);
\draw (a.corner 1) to [in=50,out=130,distance = 0.5cm,loop] (a.corner 1);
\end{tikzpicture}
\, \, \, \,   \longleftarrow 
\parbox{.20\linewidth}{
    \centering
        \includegraphics[scale=0.25]{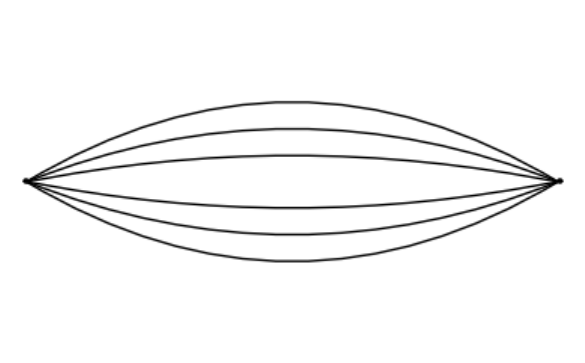}
}
\longleftarrow
\parbox{.20\linewidth}{
    \centering
        \includegraphics[scale=0.25]{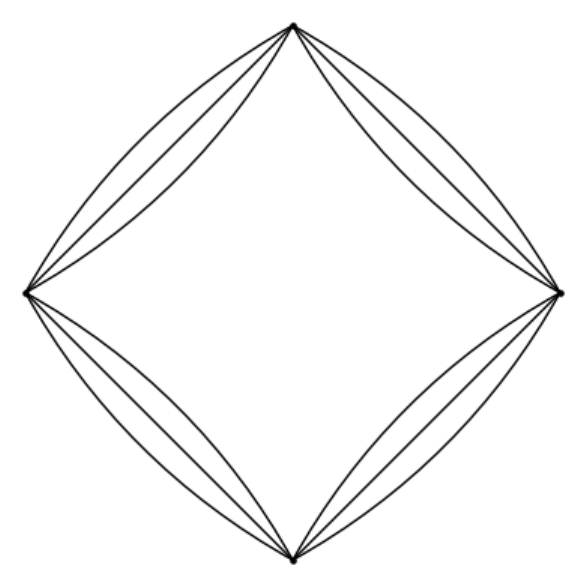}
}
\longleftarrow
\parbox{.20\linewidth}{
    \centering
        \includegraphics[scale=0.25]{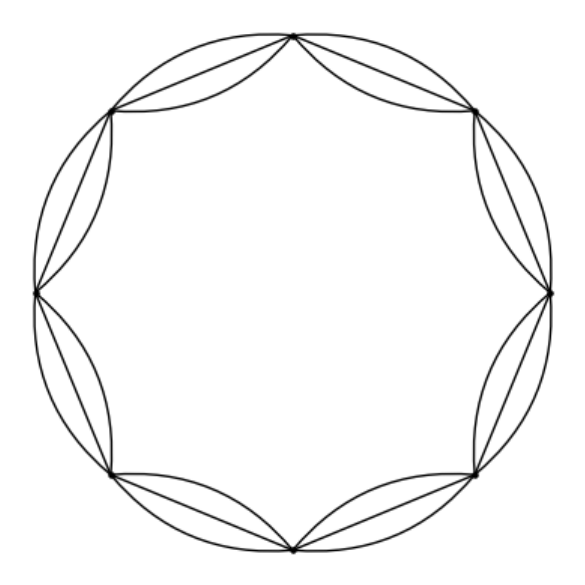}
}
\longleftarrow \ldots
\end{equation*}
The characteristic series is
$$f_{X,\alpha}(T) = -3T^{2} + 3T^{3} -\ldots \in \mathbb{Z}\llbracket T \rrbracket \subseteq \mathbb{Z}_{2}\llbracket T \rrbracket. $$
Thus
$$\mu(B_{3},\alpha) = 0 \text{ and } \lambda(B_{3},\alpha) = 1. $$
Using SageMath, we calculate
$$\kappa_{1} =2 \cdot 3, \kappa_{2} = 2^{2} \cdot 3^{3}, \kappa_{3} = 2^{3} \cdot 3^{7}, \kappa_{4} = 2^{4} \cdot 3^{15}.$$
We have
$${\rm ord}_{2}(\kappa_{n}) = n,$$
for $n \ge 0$.  Consider the dihedral group $D_{8} = \langle \rho, \tau \rangle$ (of cardinality $8$), viewed as a subgroup of the permutation group on four elements, where $\rho = (1 \; 2 \;3 \; 4)$ is of order $4$ and $\tau = (1 \; 4)(2 \; 3)$ is of order $2$.  Consider $\beta:S \rightarrow D_{8}$ given by
$$s_{1} \mapsto \rho, s_{2} \mapsto \tau, \text{ and } s_{3} \mapsto 1, $$
where $1$ is the neutral element of $D_{8}$.  Since $\beta(S)$ generates $D_{8}$, the graph $Y = X(D_{8},S,\beta)$ is connected and the cover $p:Y \rightarrow X$ is Galois with Galois group isomorphic to $D_{8}$:
\begin{equation*}
\begin{tikzpicture}[baseline={([yshift=-1.7ex] current bounding box.center)}]
\node[draw=none,minimum size=3cm,regular polygon,regular polygon sides=1] (a) {};
\foreach \x in {1}
  \fill (a.corner \x) circle[radius=0.7pt];
\draw (a.corner 1) to [in=50,out=130,loop] (a.corner 1);
\draw (a.corner 1) to [in=50,out=130,distance = 0.8cm,loop] (a.corner 1);
\draw (a.corner 1) to [in=50,out=130,distance = 0.5cm,loop] (a.corner 1);
\end{tikzpicture}
\, \, \, \,   \stackrel{p}{\longleftarrow}
\parbox{.20\linewidth}{
    \centering
        \includegraphics[scale=0.25]{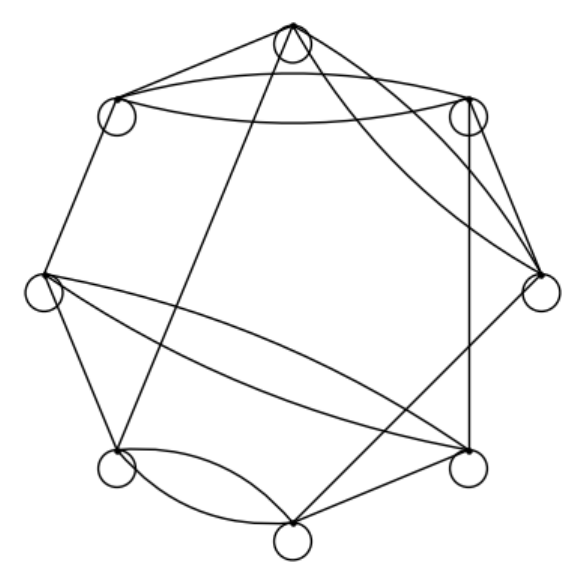}
}
\end{equation*}
An argument similar to the one made in the previous example shows that $\alpha \circ p:S \rightarrow \mathbb{Z}_{2}$ satisfies Assumption \ref{main assumption}.  We obtain the $\mathbb{Z}_{2}$-tower above $Y$:
\begin{equation*}
\parbox{.20\linewidth}{
    \centering
        \includegraphics[scale=0.25]{./ex2t_graph0.pdf}
}
\longleftarrow 
\parbox{.20\linewidth}{
    \centering
        \includegraphics[scale=0.25]{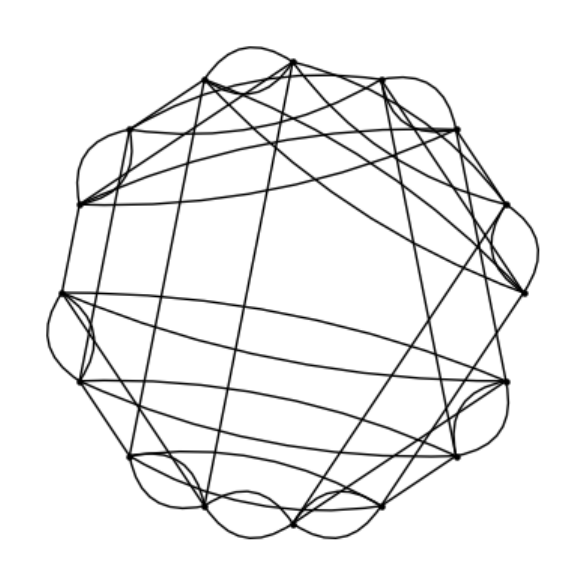}
}
\longleftarrow
\parbox{.20\linewidth}{
    \centering
        \includegraphics[scale=0.25]{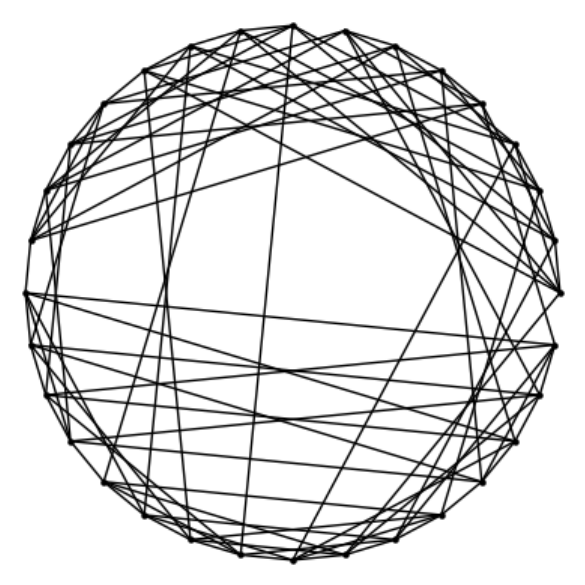}
}
\longleftarrow \, \, \ldots
\end{equation*}
We calculate the characteristic series
$$f_{Y,\alpha \circ p}(T) = - 55296T^{2} + 55296T^{3} + 39168T^{4} + \ldots  \in \mathbb{Z}\llbracket T \rrbracket \subseteq \mathbb{Z}_{2}\llbracket T \rrbracket,$$ 
for which the first coefficient that is not divisible by $2$ is the coefficient of $T^{16}$.  Thus, we obtain
$$\mu(Y,\alpha \circ p)=0 \text{ and } \lambda(Y,\alpha \circ p) = 15. $$
Using SageMath, we calculate
$$\kappa_{0} = 2^{8} \cdot 3^{2}, \kappa_{1} = 2^{21} \cdot 3^{5} \cdot 5^{2}, \kappa_{2} = 2^{48} \cdot 3^{13}\cdot 5^{2}, \kappa_{3} = 2^{63} \cdot 3^{17} \cdot 5^{2} \cdot 7^{4} \cdot 17^{6} \cdot 31^{4},$$
and 
$$\kappa_{4} = 2^{78} \cdot 3^{25} \cdot 5^{2} \cdot 7^{4} \cdot 17^{10} \cdot 31^{4} \cdot 97^{4} \cdot 113^{4} \cdot 577^{6}. $$
We have
$${\rm ord}_{2}(\kappa_{n}) = 15n + 18,$$
for $n \ge 2$.  Here, we have $[Y:X] = 8$ and the relation
$$\lambda(Y,\alpha \circ p) + 1 = [Y:X] \cdot (\lambda(B_{3},\alpha) + 1), $$
as expected by Theorem \ref{main}.

\end{example}

\section{Growth of Iwasawa invariants in noncommutative pro-$\ell$ towers}\label{section 5}
\subsection{Growth in uniform pro-$\ell$ towers} 
\par The Iwasawa theory of graphs was motivated by the study of growth properties of the number of spanning trees in $\Z_\ell$-towers of graphs. These $\Z_\ell$-towers are described by voltage assignments, as described in \S \ref{vo_assign}. Of natural interest is to study similar questions for more general pro-$\ell$ towers of graphs, subject to suitable conditions. As a first step, it seems of natural interest to consider $\Z_\ell^d$-towers of connected graphs as is done for instance in \cite{Sage/Vallieres:2022}. We study a related theme, and our results are motivated by the results of Cuoco \cite{cuoco1980growth}, who studied the growth of Iwasawa invariants in $\Z_\ell^2$-extensions of number fields. We do however work in a more general setting, which we now explain.  

\par Let $G$ be an infinite pro-$\ell$ group. The lower central $\ell$-series is defined recursively as follows $G_0:=G$, and $G_{n+1}:=\left(G_{n}\right)^\ell[G_{n}, G]$. Set $G^{(n)}$ to denote $G/G_n$. A set of elements $A_1, \dots, A_r\in G$ is a set of topological generators if the closure of the subgroup generated by them is $G$. The group $G$ is said to be finitely generated if there is a finite set of topological generators. We recall the notion of a uniform pro-$\ell$ group (cf. \cite[Definition 4.1]{dixon2003analytic}). 
\begin{definition}\label{uniform def}
The group $G$ is \emph{uniform} if the following conditions are all satisfied 
\begin{enumerate}
\item\label{uniform def 1} $G$ is finitely generated,
\item\label{uniform def 2} $[G,G]\subseteq G^\ell$ if $\ell$ is odd and $[G,G]\subseteq G^4$ if $\ell=2$, 
\item\label{uniform def 3} $[G_n:G_{n+1}]=[G:G_1]$ for all $n\geq 1$. 
\end{enumerate} 
\end{definition}

Let $d\in \Z_{\geq 1}$ be so that $[G:G_1]=\ell^d$, then, we find that $[G:G_n]=\ell^{nd}$. We refer to $d$ as the \emph{dimension} of $G$. A $G$-tower over a graph $X$ is a sequence of connected graphs related by covering maps 
\begin{equation} 
X = Y_0 \leftarrow Y_1\leftarrow \ldots \leftarrow Y_n \leftarrow \ldots,
\end{equation}
such that the map $Y_n\rightarrow X$ obtained by composing the maps $Y_n\rightarrow Y_{n-1}\rightarrow \dots \rightarrow X$ is a Galois cover with Galois group $G^{(n)}$. In \S \ref{section 5.2}, we shall construct an example for which $G$ is the pro-$\ell$ group $\op{ker}\left(\op{SL}_2(\Z_\ell)\rightarrow \op{SL}_2(\Z/\ell \Z)\right)$. Such covers can be constructed from voltage assignments. In greater detail, let $g_1, \dots, g_k$ be a set of topological generators for $G$, and choose an enumeration $s_1, \dots, s_t$ of $S$. Assume that $t\geq k$, and let $\alpha:S\rightarrow G$ be the voltage assignment that sends $s_i\mapsto g_i$ for $i\leq k$, and to any arbitrary elements of $G$ for $i>k$. Let $\alpha_{/n}:S\rightarrow G^{(n)}$ be the map obtained upon composing $\alpha$ with the quotient map $G\rightarrow G^{(n)}$. Let $Y_n$ be the graph given by $X(G^{(n)}, S, \alpha_{/n})$, and note that there are maps $Y_n\rightarrow Y_{n-1}$ for all $n\geq 1$. The map $Y_n\rightarrow X$ obtained upon composition is Galois provided $Y_n$ is connected. In our next discussion, let us assume that this is indeed the case. 
Suppose that we are given a $\Z_\ell$-tower
\begin{equation}\label{X tower last}X=X_0\leftarrow X_1\leftarrow \dots \leftarrow X_n\leftarrow \dots\end{equation} over $X$. We set $\Gamma^{(n)}:=\op{Gal}(X_n/X)$ and $\Gamma:=\varprojlim_n \Gamma^{(n)}$. Note that there are isomorphisms $\Gamma^{(n)}\simeq \Z/\ell^n\Z$ and $\Gamma\simeq  \Z_\ell$. Set $\G:=G\times \Gamma$ and $\G^{(n, m)}:=G^{(n)}\times \Gamma^{(m)}$. Furthermore, there is a voltage assignment $\beta:S\rightarrow \Gamma$ such that $X_n=X(\Gamma^{(n)}, G, \beta_{/n})$. Let $\gamma:S\rightarrow \G$ be the voltage assignment given by $\gamma(s)=(\alpha(s), \beta(s))$. Set $\gamma_{n,m}$ to denote the composite of $\gamma$ with the quotient map to $\G^{(n, m)}$, and set also $Y_{n,m}:=X(\G^{(n, m)}, S, \gamma_{n,m})$. Denote by $\cY_n$ the tower over $Y_n$ given by
\[Y_n=Y_{n,0}\leftarrow Y_{n,1}\leftarrow Y_{n, 2}\leftarrow Y_{n,3}\leftarrow \dots \leftarrow Y_{n,m}\leftarrow \dots,\]
and let $\mu_n$ (resp. $\lambda_n$) denote the $\mu$-invariant (resp. $\lambda$-invariant) for the tower $\cY_n$. Set $\mu:=\mu_0$ and $\lambda:=\lambda_0$.

\begin{theorem}\label{lastthm}
With respect to the notation above, assume that the following conditions hold
\begin{enumerate}
    \item the group $G$ is uniform of dimension $d$,
    \item the graph $Y_{n,m}$ is connected for all $m$ and $n$,
    \item the $\mu$-invariant $\mu=\mu_0$ is equal to $0$.
\end{enumerate}
Then, for all values of $n$, $\mu_n=0$, and furthermore, as a function of $n$, 
\begin{equation}\label{lastthm eqn}\lambda_n=\ell^{dn} (\lambda+1)-1.\end{equation}
\end{theorem}

\begin{proof}
The $\Z_\ell$-tower $\cY_n$ is the pullback of \eqref{X tower last} under the map $Y_n\rightarrow X$. It thus follows from Theorem \ref{thmA} that $\lambda_n+1=[Y_n:X](\lambda+1)=|G^{(n)}|(\lambda+1)=\ell^{nd}(\lambda+1)$, and the result follows.
\end{proof}

\subsection{An Example}\label{section 5.2}
We illustrate the above result through an example for which the assumptions are shown to be satisfied. Assume that $\ell$ is an odd prime. Let $G$ denote the kernel of the mod-$\ell$ reduction map 
\[G:=\op{ker}\left(\op{SL}_2(\Z_\ell)\rightarrow \op{SL}_2(\Z/\ell\Z)\right).\]
We find that for all $n$, \[G_n=\op{ker}\left(\op{SL}_2(\Z_\ell)\rightarrow \op{SL}_2(\Z/\ell^{n+1}\Z)\right)\]
and $G^{(n)}=G/G_n=\op{ker}\left(\op{SL}_2(\Z/\ell^{n+1}\Z)\rightarrow \op{SL}_2(\Z/\ell\Z)\right)$.
\par Proposition \ref{last prop} below shows that the group $G$ is uniform pro-$\ell$ with dimension $d=3$ (in the sense of Definition \ref{uniform def}). Furthermore, we shall obtain an explicit set of generators for $G$. First, we set up some preliminary notation. Let $\op{sl}_2(\Z/\ell\Z)$ denote the $3$-dimensional $\Z/\ell\Z$-vector space of $2\times 2$ square matrices with trace zero. Given a matrix $A\in \op{sl}_2(\Z/\ell\Z)$, let $\widetilde{A}$ be a lift of $A$ to a matrix with coefficients in $\Z_\ell$. The element $\ell^n A\in M_2(\Z/\ell^{n+1}\Z)$ is taken to be $\ell^n \widetilde{A}$, and is clearly independent of the choice of lift of $A$. Identify $G_n/G_{n+1}$ with the kernel of the mod-$\ell^n$ reduction map $\op{SL}_2(\Z/\ell^{n+1}\Z)\rightarrow \op{SL}_2(\Z/\ell^n\Z)$. The map \[\op{exp}_n:\op{sl}_2(\Z/\ell\Z)\rightarrow G_n/G_{n+1}\] sending $A\mapsto \op{exp}_n(A):=\op{I}+\ell^n A$ is an isomorphism of groups.
\begin{proposition}\label{last prop}
    With respect to notation above, the group $G$ is a uniform pro-$\ell$ group of dimension $d=3$ and is topologically generated by the matrices \[A_1:=\mtx{1}{\ell}{0}{1}, A_2:=\mtx{1+\ell}{0}{0}{(1+\ell)^{-1}}\text{ and }A_3:=\mtx{1}{0}{\ell}{1}.\]
\end{proposition}
\begin{proof}
Let us first verify the condition \eqref{uniform def 1} of Definition \ref{uniform def}, and show that $A_1, A_2, A_3$ is a set of generators. Let $H$ be the closure of the subgroup of $G$ generated by $A_1, A_2$ and $A_3$. Denote by $\pi_n:G\rightarrow G^{(n)}$ the natural quotient map. It suffices to show that $\pi_n(H)=G^{(n)}$ for all $n$. We prove this by induction on $n$.

\par Note that $G^{(1)}=G/G_1$ is generated by 
\[\op{exp}_1\left(\mtx{0}{1}{0}{0}\right), \op{exp}_1\left(\mtx{1}{0}{0}{-1}\right)\text{ and }\op{exp}_1\left(\mtx{0}{0}{1}{0}\right).\]The above matrices are $\pi_1(A_1), \pi_1(A_2)$ and $\pi_1(A_3)$ respectively. Therefore, $\pi_1(H)=G^{(1)}$.
\par For the inductive step, assume that $\pi_n(A_1), \pi_n(A_2)$ and $\pi_n(A_3)$ generate $G^{(n)}$. We show that $\pi_{n+1}(A_1), \pi_{n+1}(A_2)$ and $\pi_{n+1}(A_3)$ generate $G^{(n+1)}$. Identifying $G_{n-1}/G_{n}$ with $\op{sl}_2(\Z/\ell\Z)$, we find that $G_{n-1}/G_{n}$ is generated by the matrices\[\begin{split}&B_n:=\op{exp}_{n}\left(\mtx{0}{1}{0}{0}\right)=\mtx{1}{\ell^{n}}{0}{1}, \\
&C_n:=\op{exp}_{n}\left(\mtx{1}{0}{0}{-1}\right)=\mtx{1+\ell^{n}}{0}{0}{1-\ell^{n}}, \\
&D_n:=\op{exp}_{n}\left(\mtx{0}{1}{0}{1}\right)=\mtx{1}{\ell^{n}}{0}{1}\\
\end{split}\]
which are considered modulo $G_{n}$. Note that since $\pi_{n}(H)=G^{(n)}$, the matrices $B_n$, $C_n$ and $D_n$ are contained in $\pi_n(H)$. Therefore, we may choose lifts $\tilde{B}_n, \tilde{C}_n, \tilde{D}_n$ of $B_n$, $C_n$ and $D_n$ respectively to $\pi_{n+1}(H)$. The following relations hold in $G^{(n+1)}$
\[\begin{split}& A_2 \tilde{B}_n A_2^{-1} \tilde{B}_n^{-1}=\mtx{1}{2\ell^{n+1}}{0}{1}\\
& A_2 \tilde{D}_n A_2^{-1} \tilde{D}_n^{-1}=\mtx{1}{0}{-2\ell^{n+1}}{1}\\
& A_1 \tilde{D}_n A_1^{-1} \tilde{D}_n^{-1}=\mtx{1+\ell^{n+1}}{0}{0}{1-\ell^{n+1}},\\
\end{split}\]
see the proof of \cite[Lemma 2.8]{ray2021constructing}. The above matrices are therefore contained in $\pi_{n+1}(H)$. Since $\ell$ is assumed to be odd, we find that they generate $G_n/G_{n+1}=\op{exp}_n\left(\op{sl}_2(\Z/\ell\Z)\right)$. It follows that $\pi_{n+1}(H)$ contains $G_n/G_{n+1}$. Since $\pi_n(H)=G/G_n$ by assumption, and $\pi_{n+1}(H)$ contains $G_n/G_{n+1}$ it follows that $\pi_{n+1}(H)=G^{(n+1)}$. This completes the inductive step. We conclude that $H=G$, and therefore, $G$ is topologically generated by $A_1, A_2$ and $A_3$. In particular, \eqref{uniform def 1} of Definition \ref{uniform def} is satisfied. 
\par The same arguments as above show that the matrices
\[\begin{split}& A_1^\ell=\mtx{1}{\ell^2}{0}{1},\\
& A_2^{\ell}=\mtx{(1+\ell^2-\frac{1}{2}\ell^2\dots)}{0}{0}{(1-\ell^2-\frac{1}{2}\ell^2\dots)},\\
& A_3^\ell=\mtx{1}{0}{\ell^2}{1}
\end{split}\]
generate $G_2$. Thus the property \eqref{uniform def 2} of Definition \ref{uniform def} is satisfied. Finally, since the exponential map 
\[\op{exp}_n: \op{sl}_2(\Z/\ell\Z)\rightarrow G_n/G_{n+1}\] is seen to be an isomorphism for all $n$, it follows that the property \eqref{uniform def 3} is satisfied as well.
\end{proof}

Let $X$ be the bouquet with $4$ undirected edges. Let $s_1, \dots, s_4$ denote the edges in $S$, and $\alpha:S\rightarrow G$ denote the map taking $s_i\mapsto A_i$ for $i\leq 3$ and $s_4\mapsto 1$. Let $\beta:S\rightarrow \Gamma$ denote the map taking $s_i\mapsto 0$ for $i\leq 3$ and $s_4\mapsto 1$. Then, the map $\gamma:S\rightarrow \G=G\times \Gamma$ maps $s_1, \dots, s_4$ to a set of generators, moreover, all graphs $Y_{n, m}$ are connected since $\gamma_{n,m}(S)$ generates $G^{(n)} \times \Gamma^{(m)}$ (cf. Theorem \ref{connectedness}). Consider the tower 
\[X=X_0\leftarrow X_1 \leftarrow \dots\leftarrow X_n\leftarrow \dots ,\]
where $X_n$ is taken to be $X(\Gamma^{(n)}, S, \beta_{/n})$. We see that the characteristic series is 
$$f_{X,\beta}(T) = -T^2+T^3-T^4+\dots,$$ 
and hence, $\mu=0$ and $\lambda=1$.  Recall that $\mu_n$ and $\lambda_n$ are the $\mu$ and $\lambda$-invariants of the tower $\Y_n$. According to the formula \eqref{lastthm eqn} of Theorem \ref{lastthm}, we find that $\mu_n=0$ for all $n$, and that $\lambda_n=2\ell^{3n}-1$ for all $n$.
\bibliographystyle{alpha}
\bibliography{references}
\end{document}